\documentclass[10pt]{amsart}
\usepackage{amsmath}
\usepackage{amsfonts}
\usepackage{amssymb}
\usepackage{amsthm}
\usepackage{url}
\usepackage{tikz-cd}
\usepackage{dsfont}
\usepackage{graphicx}
\usepackage{caption}
\usepackage{subcaption}
\usepackage{comment} 
\usepackage{stmaryrd}
\usepackage{hyperref}
\usepackage{todonotes}
\usepackage{color}
\usepackage{enumerate}

\numberwithin{equation}{section}

\newtheorem{proposition}{Proposition}[section]
\newtheorem{lemma}[proposition]{Lemma}
\newtheorem{theorem}[proposition]{Theorem}
\newtheorem{corollary}[proposition]{Corollary}
\newtheorem{conjecture}{Conjecture}[section]

\theoremstyle{definition}
\newtheorem{remark}[proposition]{Remark}
\newtheorem{definition}[proposition]{Definition}

\DeclareMathOperator{\Id}{Id}

\DeclareMathOperator{\Aut}{Aut}
\DeclareMathOperator{\Ric}{Ric}

\DeclareMathOperator{\V}{V}
\DeclareMathOperator{\Isom}{Isom}

\DeclareMathOperator{\Lie}{Lie}

\renewcommand{\L}{\mathcal{L}}
\renewcommand{\phi}{\varphi}
\DeclareMathOperator{\Stab}{Stab}
\newcommand{\scG}{\mathcal{G}}
\newcommand{\scP}{\mathcal{P}}
\newcommand{\scQ}{\mathcal{Q}}

\newcommand{\R}{\mathbb{R}}
\newcommand{\C}{\mathbb{C}}

\newcommand{\G}{\mathcal{G}}
\newcommand{\pr}{\mathbb{P}}
\renewcommand{\epsilon}{\varepsilon}
\newcommand{\scH}{\mathcal{H}}
\newcommand{\D}{\mathcal{D}}
\newcommand{\M}{\mathcal{M}}
\newcommand{\scN}{\mathcal{N}}
\newcommand{\scO}{\mathcal{O}}

\newcommand{\ddb}{i\partial \bar\partial}
\newcommand{\scL}{\mathcal{L}}
\newcommand{\mfh}{\mathfrak{h}}
\newcommand{\mft}{\mathfrak{t}}
\newcommand{\mfg}{\mathfrak{g}}
\newcommand{\mfp}{\mathfrak{p}}

\newcommand{\F}{\mathcal{F}}
\newcommand{\scB}{\mathcal{B}}

\newcommand{\mfk}{\mathfrak{k}}

\newcommand{\scD}{\mathcal{D}}
\newcommand{\scV}{\mathcal{V}}
\newcommand{\scR}{\mathcal{R}}

\newcommand{\ddc}{i\partial\bar\partial}
\newcommand{\db}{\mathcal{\bar\partial}}

\title[Optimal symplectic connections on holomorphic submersions]{Optimal symplectic connections on holomorphic submersions}

\author[Ruadha\'i Dervan and Lars Martin Sektnan]{Ruadha\'i Dervan and Lars Martin Sektnan}
\address{Ruadha\'i Dervan, DPMMS, Centre for Mathematical Sciences, Wilberforce Road, Cambridge CB3 0WB, United Kingdom}
\email{R.Dervan@dpmms.cam.ac.uk}

\address{Lars Martin Sektnan, D\'epartement de math\'ematiques, Universit\'e du Qu\'ebec \`a Montr\'eal, Case postale 8888, succursale
centre-ville, Montr\'eal (Qu\'ebec), H3C 3P8, Canada}
\email{lars.sektnan@cirget.ca}

\begin{document}

\begin{abstract}  

The main result of this paper gives a new construction of extremal K\"ahler metrics on the total space of certain holomorphic submersions, giving a vast generalisation and unification of results of Hong, Fine and others. The principal new ingredient is a novel geometric partial differential equation on such fibrations, which we call the optimal symplectic connection equation. 

We begin with a smooth fibration for which all fibres admit a constant scalar curvature K\"ahler metric. When the fibres admit automorphisms, such metrics are not unique in general, but rather are unique up to the action of the automorphism group of each fibre. We define an equation which, at least conjecturally, determines a canonical choice of constant scalar curvature K\"ahler metric on each fibre. When the fibration is a projective bundle, this equation specialises to asking that the hermitian metric determining the fibrewise Fubini-Study metric is Hermite-Einstein. 

Assuming the existence of an optimal symplectic connection, and the existence of an appropriate twisted extremal metric on the base of the fibration, we show that the total space of the fibration itself admits an extremal metric for certain polarisations making the fibres small.

\end{abstract}

\maketitle


\section{Introduction}

A hermitian metric $h$ on a holomorphic vector bundle $E$ over a compact K\"ahler manifold $(B,\omega_B)$ induces a hermitian metric on the line bundle $\scO_{\pr(E)}(1)$ over the projectivisation $\pr(E)$ of $E$ with curvature  $\omega_h \in c_1(\scO_{\pr(E)}(1))$. On each fibre $\pr(E)_b$ of $\pi: \pr(E) \to B$ over $b \in B$, $\omega_h$ restricts to a Fubini-Study metric. A Fubini-Study metric on projective space is not unique, but rather is unique up to the automorphisms of projective space. In this way, one can think of a hermitian metric on a vector bundle as a choice of Fubini-Study metric on each fibre on the associated projective bundle. Going one step further, one can ask for a hermitian metric to solve the \emph{Hermite-Einstein}\ equation: if $F_h$ is the curvature of $h$, this equation is \begin{equation}\label{intro-HE}\Lambda_{\omega_B} F_h = \lambda \Id,\end{equation} where $\lambda$ is the appropriate topological constant. Hermite-Einstein metrics are unique, when they exist, and so from our perspective give a \emph{canonical} choice of fibrewise Fubini-Study metric on $\pr(E)$ \cite{donaldson-surfaces}.

The goal of this paper is to introduce a vast generalisation of the notion of a Hermite-Einstein metric which makes sense on a much wider class of smooth fibrations (that is, holomorphic submersions) $\pi: X \to B$, and to use this to give a new construction of extremal K\"ahler metrics on the total space $X$ of such fibrations. This equation, which we call the \emph{optimal symplectic connection equation}, makes sense for fibrations for which each fibre admits a constant scalar curvature K\"ahler metric. We expect the theory of this new analogue of the Hermite-Einstein condition to be as deep as that of Hermite-Einstein metrics; in particular, we conjecture that the existence of solutions of this partial differential equation is equivalent to an algebro-geometric stability condition, extending the classical notion of slope stability of vector bundles.

\subsection{Optimal symplectic connections} We begin by explaining the new equation. The starting point is the observation that a Fubini-Study metric $\omega_{FS}$ on projective space $\pr^m$ is \emph{K\"ahler-Einstein}: it satisfies $\Ric\omega_{FS} = (m+1)\omega_{FS}$. The K\"ahler-Einstein condition makes sense on manifolds with positive, negative or trivial first Chern class, and the appropriate generalisation to arbitrary K\"ahler manifolds is to condsider \emph{constant scalar curvature K\"ahler} (cscK) metrics. 

Now let $\pi: X\to B$ be a smooth fibration such that $H$ is a relatively ample line bundle on $X$ and $L$ is an ample line bundle on $B$. We allow the fibres to be non-biholomorphic.  Suppose $\omega_X \in c_1(H)$ is a relatively K\"ahler metric which is cscK on each fibre $(X_b,H_b)$, and let $\omega_B \in c_1(L)$ be a K\"ahler metric on $B$. When the fibres admit automorphisms, such a choice of $\omega_X$ is not unique (even up to the pullback of a form from $B$), but rather is unique up to fibrewise automorphisms, mirroring the situation for projective bundles \cite{donaldson-scalar-curvature,berman-berndtsson}. Thus a natural question is whether or not there is a \emph{canonical} choice of such $\omega_X$. The answer to this question will be phrased in terms of various geometric objects associated to $\omega_X$; we now briefly describe the necessary terminology. More precise details are given in Section \ref{sec:optimal}.

Since $\pi: X \to B$ is a smooth fibration, the vertical bundle $\scV = \ker d\pi$ is a holomorphic vector bundle. The relatively K\"ahler metric $\omega_X$ determines a horizontal subbundle $\scH \subset TX$ via $$\scH_x = \{ u \in T_xX\ |\ \omega_0(u,v) = 0 \textrm{ for all } v \in \scV\}.$$ In this context, $\omega_X$ is usually called a \emph{symplectic connection} \cite[Section 6]{symplectic-topology}. There is a notion of the \emph{symplectic curvature} $F_{\scH}$ of $\omega_X$, which is simply the curvature of the associated Ehresmann connection. Using the fibrewise comoment map $\mu^*$, one can think of $\mu^*(F_{\scH})$ as a two-form on $B$ with values in functions on the fibres. The relatively K\"ahler metric $\omega_X$ also induces a metric on the top exterior power $\wedge^m\scV$, whose curvature we denote $\rho$. The restriction of $\rho$ to any fibre $X_b$ is the Ricci curvature of $\omega_X|_{X_b}$, but $\rho$ will also have a horizontal component $\rho_{\scH}$ under the horizontal-vertical decomposition of forms induced by $\omega_X$. On each fibre $\omega_X$ determines a Laplacian operator; piecing these together fibrewise determines a vertical Laplacian operator $\Delta_{\scV}$ on functions on $X$.

The main additional assumption we make is that the dimension of the automorphism group $\Aut(X_b,H_b)$ of the fibre $(X_b,H_b)$ is independent of $b \in B$. Under this assumption, one can form a smooth vector bundle $E$ whose fibres $E_b$ consist of functions which are holomorphy potentials on $X_b$: functions $\phi \in C^{\infty}(B)$  which satisfy the equation $\bar\partial \nabla^{1,0}\phi =0$, which simply says that (the $(1,0)$-part of) their gradient is a holomorphic vector field. The relatively K\"ahler metric $\omega_X$ determines a fibrewise $L^2$-inner product on functions on $X$, which in turn defines an orthogonal projection $p: C^{\infty}(X) \to E$.

The final notation needed is the horizontal contraction $\Lambda_{\omega_B}$, which sends horizontal two-forms to functions by contracting with respect to $\omega_B$. With all of this in place, we say that $\omega_X$ defines an \emph{optimal symplectic connection} if  \begin{equation}\label{eq:intro-equation}p(\Delta_{\scV}(\Lambda_{\omega_B} \mu^*(F_{\scH})) + \Lambda_{\omega_B} \rho_{\scH})=0.\end{equation} The first remark to make is that, although not obvious at first glance, this equation is indeed equivalent to the Hermite-Einstein condition of equation \eqref{intro-HE} on projective bundles. The second remark is that, fixing any reference $\omega_X$, by uniqueness of cscK metrics up to automorphisms, any other fibrewise cscK metric is determined by a fibrewise holomorphic vector field, and one can identify such an object with a smooth global section of $E$. Thus one can think of this equation as a nonlinear partial differential equation on the bundle $E$. The third remark is that, just as the constant scalar curvature equation simplifies to the K\"ahler-Einstein condition on Fano manifolds, we shall show that the optimal symplectic connection equation reduces to an equation only involving the symplectic curvature $$p(\Delta_{\scV}(\Lambda_{\omega_B} \mu^*(F_{\scH})) + \Lambda_{\omega_B}  \mu^*(F_{\scH}))=0$$ on Fano fibrations, which is to say when all fibres are Fano manifolds. The main result regarding this equation that we establish in the present work concerns its linearisation, which we show is a very natural second order elliptic operator on the bundle $E$, whose kernel can be thought of as fibrewise holomorphic vector fields which are \emph{globally} holomorphic on $X$. 

Before proceeding to the main result of this paper, a construction of extremal metrics on the total space of such fibrations $\pi: X \to B$, we formulate some conjectures regarding existence and uniqueness of optimal symplectic connections. The first, and most important, relates to algebraic geometry:

\begin{conjecture}\label{Hitchin-Kobayashi} A fibration $\pi: (X,H) \to (B,L)$ admits an optimal symplectic connection if and only if it satisfies some algebro-geometric stability condition. \end{conjecture}

\noindent This should be thought of an analogue of the Hitchin-Kobayashi correspondence of Donaldson-Uhlenbeck-Yau \cite{donaldson-infinite,uhlenbeck-yau}, which relates the existence of Hermite-Einstein metrics to slope stability of the bundle. In particular, we expect that the existence of solutions of the optimal symplectic connection equation should be independent of choice of $\omega_B$ in its class. The next main conjecture relates to uniqueness:

\begin{conjecture} If $\omega_X$ and $\omega_X'$ are two optimal symplectic connections, then up to pullback from $B$ they are related by an automorphism $g \in \Aut(X,H)$ which satisfies $\pi \circ g = \pi$: there is a form $\beta$ on $B$ such that $$g^*\omega_X = \omega_X' + \pi^*\beta.$$ \end{conjecture}

\noindent Thus, optimal symplectic connections should, as desired, give a canonical choice of fibrewise cscK metric on fibrations. 

Finally, we conjecture that the optimal symplectic connection condition arises from some infinite dimensional moment map picture, in a similar manner as to how the Hermite-Einstein condition arises. We plan to return to each of these topics in the future.

\subsection*{Extremal metrics on fibrations} One of the most important problems in K\"ahler geometry is to understand the existence of extremal metrics on K\"ahler manifolds, which give a canonical choice of K\"ahler metric when they exist. The existence of such metrics is a very challenging problem, closely related to algebraic geometry through the Yau-Tian-Donaldson conjecture \cite{donaldson-survey}, and there are still rather few constructions of extremal metrics available. 

We now explain how we use optimal symplectic connections to produce extremal metrics on the total space of fibrations. We shall state three theorems in this direction, of increasing generality, and beginning with the simplest case of fibrations $\pi: (X,H) \to (B,L)$ such that all fibres are isomorphic,  which is the situation closest to that of projective bundles. In addition we assume that the base $(B,L)$ as well as the total space $(X,H)$ admit no automorphisms. 

\begin{theorem}\label{intro:hongstyle} Suppose $(B,L)$ admits a cscK metric $\omega_B$ and $\pi: (X,H) \to (B,L)$ admits an optimal symplectic connection $\omega_X$ solving equation \eqref{eq:intro-equation}. Then there exists a cscK metric in the class $kc_1(L)+c_1(H)$ for all $k \gg 0$. \end{theorem}

\noindent As mentioned above, the optimal symplectic connection equation specialises to the Hermite-Einstein condition on projective bundles, and so this result generalises work of Hong \cite{hong1}. When the fibres are integral coadjoint orbits, with the fibration induced by a holomorphic principal bundle, we similarly show that the optimal symplectic connection equation reduces to a Hermite-Einstein equation on the principal bundle. This produces new examples of cscK manifolds, and answers a question of Br\"onnle  \cite[Section 2.3]{bronnle}. 

When the fibres are not biholomorphic, the statement involves some moduli theory. Since each fibre is assumed to admit a cscK metric, there is a moduli space $\M$ of cscK manifolds together with a moduli map $q: B \to \M$ such that $b\in B$ maps to the point associated to $(X_b,L_b)$ in $\M$. When the fibres have \emph{discrete} automorphism group, the existence of $\M$ is a classical result of Fujiki-Schumacher \cite{fujiki-schumacher}. In the general case, the existence of such a moduli space is a recent result of the first author and Naumann \cite{DN} (see also \cite{lwx,odaka-moduli} for the special case of K\"ahler-Einstein Fano manifolds). The moduli space $\M$ carries a Weil-Petersson type K\"ahler metric $\Omega_{WP}$, and we denote $\alpha = q^*\Omega_{WP}$ \cite{fujiki-schumacher,DN}, which is a smooth semi-positive form on $B$. Our assumption on $(B,L)$ is that there is a solution $\omega_B \in c_1(L)$ of the \emph{twisted cscK equation} $$S(\omega_B) - \Lambda_{\omega_B} \alpha = const,$$ and that the automorphism groups of the moduli map $\Aut(q) = \{g \in \Aut(B,L): q \circ g = q\}$ and the manifold $(X,H)$ are discrete.

\begin{theorem}\label{intro:finestyle} Suppose $(B,L)$ admits a twisted cscK metric $\omega_B$ as above and $\pi: (X,H) \to (B,L)$ admits an optimal symplectic connection $\omega_X$ solving equation \eqref{eq:intro-equation}. Then there exists a cscK metric in the class $kc_1(L)+c_1(H)$ for all $k \gg 0$. \end{theorem}

\noindent This result extends work of Fine \cite{fine1} in the case that the fibres all have discrete automorphism group, and also generalises Theorem \ref{intro:hongstyle}. In Fine's case, the cscK metrics on the fibres are unique, and so the optimal symplectic connection condition is vacuous. Strictly speaking, the result of Fine also requires an assumption on the invertibility of the linearisation of the twisted scalar curvature: this was removed by the authors \cite{fibrations}.

We now turn to the case when the moduli map and $(X,H)$ itself are allowed to admit automorphisms. In this situation, it is more general and appropriate to consider extremal metrics. Recall that a K\"ahler metric $\omega$ is \emph{extremal} if its scalar curvature $S(\omega)$ is a holomorphy potential: $$\bar\partial\nabla^{(1,0)}S(\omega) = 0.$$ There is a corresponding notion of a twisted extremal metric, where one requires that the function $S(\omega_B) - \Lambda_{\omega_B} \alpha$ is a holomorphy potential on $B$. The extremal analogue of the optimal symplectic connection is what we call an \emph{extremal symplectic connection}: these satisfy the condition that the function $$p(\Delta_{\scV}(\Lambda_{\omega_B} \mu^*(F_{\scH})) + \Lambda_{\omega_B} \rho_{\scH})$$ on $X$ is the holomorphy potential of a global holomorphic vector field on $X$, in a sense which will be made more precise in Section \ref{sec:optimal}. This condition generalises the condition that a holomorphic vector bundle is the direct sum of two different slope stable subbundles potentially of different slopes, so that it admits a direct sum of Hermite-Einstein metrics. It will be necessary to assume that $\omega_X$ is invariant under a certain compact group of automorphisms of $(X,H)$; we refer to Remark \ref{newtorus} for the precise requirement. Here we simply mention that we expect that this requirement is automatically satisfied for any extremal symplectic connection, and that when $\omega_X$ is an optimal symplectic connection and $\omega_B$ is a twisted cscK metric, the invariance condition is vacuous.

\begin{theorem} Suppose $(B,L)$ admits a twisted extremal metric $\omega_B$ as above and all automorphisms $\Aut(q)$ of the moduli map lift to automorphisms of $(X,H)$. Suppose in addition $\pi: (X,H) \to (B,L)$ admits an extremal symplectic connection $\omega_X$. Then there exists an extremal metric in the class $kc_1(L)+c_1(H)$ for all $k \gg 0$. \end{theorem}

\noindent It follows that $(X,kL+H)$ admits a cscK metric if and only if the Futaki invariant vanishes. The situation where $X$ is the projectivisation of a direct sum of vector bundles was considered by Br\"onnle \cite{bronnle}, so that the extremal symplectic connection is induced by the Hermite-Einstein metrics, who proved the above in that situation when the base $(B,L)$ admits no automorphisms. When $E$ is simple (so the extremal symplectic connection is necessarily induced by a genuine Hermite-Einstein metric), but the base is allowed to admit an extremal metric, this is due to Lu-Seyyedali \cite{lu-seyyedali}. In the case that the fibres all have discrete automorphism group, so that the extremal symplectic connection is vacuous, this is a result of the authors \cite{fibrations}. Our result generalises all of these results, in particular it is new even in the case of projective bundles. A main point of our previous work \cite{fibrations} was that when the fibres have discrete automorphism group, one can use some moduli-theoretic arguments to show that the hypothesis that all automorphisms $\Aut(q)$ lift to $(X,H)$ is automatically fulfilled. In the situation under consideration here, this \emph{is} necessary, as one can see already for projective bundles, where the necessary condition is that the bundle is $\Aut(B,L)$-equivariant. 

It is natural to ask if there is an analogue of these results when the fibres admit extremal metrics rather than cscK metrics. We expect this to be the case, however such a result would require new ideas in the moduli theory of extremal K\"ahler manifolds.

We expect our results to be almost sharp, in the following sense. Through Conjecture \ref{Hitchin-Kobayashi}, if the fibration does not admit an optimal symplectic connection, we expect that there should be some algebro-geometric object which detects this non-existence. For projective bundles, this object would be a subsheaf, and if the bundle is slope unstable, then the projectivisation is also K-unstable for the polarisations we consider \cite{ross-thomas}, so cannot admit a cscK metric. An analogous phenomenon would imply that the hypothesis on the existence of an optimal symplectic connection is almost sharp. We remark that the requirement on the existence of a twisted cscK metric on the base is known to be essentially sharp for a similar reason: if the moduli map is K-unstable (in a stacky sense), then it cannot admit a twisted cscK metric \cite{twisted}, and moreover the total space of the fibration $(X,kL+H)$ is K-unstable for $k \gg 0$ \cite[Theorem 1.3 and Remark 4.8]{stablemaps}, so $(X,kL+H)$ itself cannot admit a cscK metric \cite{donaldson-semistable}

The proof of our result uses an adiabatic limit technique. We begin with the twisted extremal metric $\omega_B$ and the optimal symplectic connection $\omega_X$ (or an extremal symplectic connection). Since $\omega_X$ is cscK when restricted to each fibre, the K\"ahler metric $k\omega_B +\omega_X$ (suppressing pullbacks) is an extremal metric to first order in $k$ as $k$ tends to infinity. To obtain an extremal metric to leading two orders requires us to use that the base metric is twisted extremal, and that the form $\omega_X$ is an optimal symplectic connection. From here, elliptic PDE theory comes into play. To inductively produce an extremal metric to any given order requires us to give a detailed understanding of the linearisation of the optimal symplectic curvature equation at a solution. We show that this can be thought of as a second order elliptic operator on the bundle $E$ of fibrewise holomorphy potentials, whose kernel can be identified with holomorphic vector fields lifting to $X$. This, in addition to the geometric understanding of the linearisation of the twisted extremal operator obtained in \cite{fibrations}, allows us to produce an extremal metric to all orders in $k$ using a delicate inductive procedure. The final step is to perturb to a genuine extremal metric, for which we use a quantitative version of the implicit function argument as in the previous work \cite{hong1,fine1, bronnle, lu-seyyedali,fibrations}.

\vspace{4mm} \noindent {\bf Outline:} We begin in Section \ref{sec:preliminaries} with some preliminary material on extremal and twisted extremal metrics, together with the aspects of moduli theory which we shall make use of. In Section \ref{sec:optimal} we develop the basic theory of optimal symplectic connections. The core of the proof, namely the study of the linearisation of the optimal symplectic operator and the inductive argument, is contained in Section \ref{sec:approx}. Finally, we apply the implicit function theorem to produce a genuine extremal metric in Section \ref{sec:implicit}.


\vspace{4mm} \noindent {\bf Notation:} Throughout $\pi: (X,H) \to (B,L)$ will denote a smooth fibration (that is, holomorphic submersion) $\pi: X \to B$ between compact K\"ahler manifolds, with $H$ a relatively ample line bundle on $X$ and $L$ an ample line bundle on $B$. $\pi$ will always have relative dimension $m$ and $B$ will have dimension $n$. All of our results go through in the non-projective case with $L$ and $H$ replaced by a relatively K\"ahler class and a K\"ahler class respectively, however for notational convenience we use the projective notation. The only exception is Section \ref{coadjoint} on the example of fibrations with fibres coadjoint orbits, where it seems necessary to require these to be integral coadjoint orbits. We typically suppress pullbacks via $\pi$, so that if $\omega_B$ is a K\"ahler metric on $B$, its pullback to $X$ will also be denoted $\omega_B$. We will typically choose a twisted extremal or twisted cscK metric $\omega_B \in c_1(L)$, with twisting a Weil-Petersson type metric, and also a relatively cscK metric $\omega_X \in c_1(H)$.

\section{Preliminaries}\label{sec:preliminaries} We summarise the foundational results regarding extremal metrics, fibrations and moduli theory that we shall require. The material on extremal metrics and holomorphic vector fields can be found in \cite{szekelyhidi-book}.

\subsection{Holomorphic vector fields}

Let $X$ be a compact complex manifold with $H$ an ample line bundle. We call $(X,H)$ a polarised manifold. We denote by $\mfh$ the Lie algebra of holomorphic vector fields on $X$ which vanish somewhere. The automorphism group $\Aut(X)$ of $X$ is a complex Lie group, and we denote by $\Aut(X,H)\subset \Aut(X)$ the Lie subgroup with Lie algebra $\mfh$. This can be seen as the group of automorphisms of $X$ which lift to $H$ \cite{lebrun-simanca}. We will also denote by $\Aut_0(X,H)$ the connected component of the identity.

Let $\omega \in c_1(H)$ be a K\"ahler metric with corresponding Riemannian metric $g$. Any holomorphic vector field $\xi \in \mfh$ can be written in the form $$\xi^j = g^{j\bar k} \partial_{\bar k} f$$ for some $f \in C^{\infty}(X,\C)$. We call such an $f$ a \emph{holomorphy potential}. Holomorphy potentials are unique up to the addition of a constant, and so if we denote by $\bar \mfh$ the vector space of holomorphy potentials, we have $$\dim \bar \mfh = \dim \mfh + 1.$$

Denote by $\scD_{\omega}: C^{\infty}(X,\C) \to \Omega^1(TX^{1,0})$ the operator $$\scD_{\omega}(f) = \bar \partial \nabla^{1,0}f.$$ The \emph{Lichnerowicz operator} is defined to be $\scD_{\omega}^*\scD_{\omega}$, where $\scD_{\omega}^*$ denotes the adjoint with respect to the $L^2$-inner product. Excplicitly, this operator is given as \begin{equation}\label{eqn:lich}\mathcal{D}_{\omega}^* \mathcal{D}_{\omega} (\phi ) = \Delta_{\omega}^2 (\phi) + \langle \textnormal{Ric} (\omega) , \ddc \phi \rangle + \langle \nabla S (\omega), \nabla \phi \rangle,\end{equation} see \cite[p. 59]{szekelyhidi-book}. This is clearly a fourth order elliptic operator as its leading order term is $\Delta^2$. The kernel of the Lichnerowicz operator $\ker \scD_{\omega} = \ker \scD_{\omega}^*\scD_{\omega}$ then precisely corresponds to holomorphy potentials \cite[Definition 4.3]{szekelyhidi-book}, that is $$\ker \scD_{\omega}^*\scD_{\omega} = \bar \mfh.$$ Denote by $$C^{\infty}_0(X,\C) = \left\{f \in C^{\infty}(X,\C): \int_{X} f\omega^n = 0\right\}$$ the functions which integrate to zero over $X$, and let $\ker \scD_{\omega}^*\scD_{\omega}$ denote the kernel of $\scD_{\omega}^*\scD_{\omega}$ on this subspace. We then obtain a natural isomorphism $$\ker _0\scD_{\omega}^*\scD_{\omega} \cong \mfh.$$

\subsection{Extremal K\"ahler metrics} The \emph{Ricci curvature} of a K\"ahler metric $\omega$ can be defined as $$(2\pi) \Ric \omega = {-i}\partial\bar\partial \log \omega^n,$$ where one views $\omega^n$ as a Hermitian metric on $-K_X = \det TX^{1,0}$. In turn, one defines the \emph{scalar curvature} $$S(\omega) = \Lambda_{\omega} \Ric \omega$$ as the contraction of the Ricci curvature. 

\begin{definition} A K\"ahler metric $\omega$ is \begin{enumerate}[(i)]
\item \emph{K\"ahler-Einstein} if $\Ric\omega = \lambda \omega$ for some $\lambda \in \R$;
 \item a \emph{constant scalar curvature K\"ahler metric} (or \emph{cscK}) if $S(\omega)$ is constant; .
 \item \emph{extremal} if $S(\omega)$ is a holomorphy potential, which is to say $\scD_{\omega} S(\omega) = 0$.

 \end{enumerate}

 \end{definition}

\begin{remark}\label{KEorcscK}It is clear that every cscK metric is extremal, and that conversely every extremal metric is cscK if $\Aut(X,H)$ is discrete. Furthermore, every K\"ahler-Einstein metric is cscK, and if $\lambda \alpha = c_1(X) = c_1(-K_X)$ Hodge theory implies that every cscK metric is K\"ahler-Einstein \cite[p. 60]{szekelyhidi-book}.
\end{remark}

A fundamental result which we will require is that extremal metrics are actually unique in their K\"ahler class up to automorphisms, in the following sense.

\begin{theorem}\cite{donaldson-scalar-curvature,berman-berndtsson}\label{thm:uniqueness} If $\omega, \omega' \in c_1(H)$ are extremal, then there is a $g\in \Aut_0(X,H)$ with $\omega = g^*\omega'$. \end{theorem} \noindent As a simple example, the only K\"ahler-Einstein metrics on $\pr^n$ are the Fubini-Study metrics, which are unique up to automorphisms of $\pr^n$.

Let $\G$ be the linearisation of the scalar curvature operator, in the sense that if $\omega_t = \omega+ t\ddc \varphi$ is a family of K\"ahler metrics, then $$\G (\varphi) = \frac{dS(\omega_t)}{dt}\bigg\rvert_{t=0} \in C^{\infty}(X,\R).$$ Explicitly, one calculates \cite[Lemma 4.4]{szekelyhidi-book} $$\G (\varphi) = -\scD_{\omega}^*\scD_{\omega} + \frac{1}{2} \langle \nabla S (\omega), \nabla \phi \rangle.$$ Thus if $\omega$ is a cscK metric, this equals the Lichnerowicz operator.

Using the language of holomorphy potentials, the extremal condition can be written as requiring $$S(\omega) - f = 0,$$ for some $f \in \bar\mfh$ with corresponding vector field $\nu$. If we change $\omega$ to $\omega+\ddc \varphi$, then the holomorphy potential $f$ for $\nu$ changes by \begin{equation}\label{extremal-eqn}\nu(\varphi) = \frac{1}{2}\langle \nabla f, \nabla \phi\rangle,\end{equation} where the gradients and inner product are defined using $\omega$ (see for example \cite[Lemma 20]{bronnle}). Thus to find an extremal metric in $\beta$ requires us to find a $\varphi\in C^{\infty}(X,\R)$ such that $$S(\omega+\ddb \varphi) -  \frac{1}{2}\langle \nabla f, \nabla \phi\rangle - f=0.$$ The linearisation of this equation is obtained from $\scG$, as $\phi \to  \frac{1}{2}\langle \nabla f, \nabla \phi\rangle$ is linear. Explicitly, viewing the extremal operator as before as an operator $C^{\infty}(X,\R) \times \bar\mfh \to C^{\infty}(X,\R)$, and denoting by $\scL$ its linearisation, for $\omega$ an extremal metric and $f= S(\omega)$ we have $\scL(\phi,0) = -\scD_{\omega}^*\scD_{\omega}\phi.$

The final point which we shall require concerns the automorphism invariance of extremal metrics. An extremal metric $\omega$ has an associated holomorphic vector field $\nu$, and we define  $\Aut_{\nu}(X,H)\subset \Aut(X,H)$ to be the Lie subgroup of automorphisms commuting with the flow of $\nu$. Thus the Lie algebra of $\Aut_{\nu}(X,H)$ precisely consists of holomorphic vector fields which commute with $\nu$.

\begin{theorem}\label{thm:torus-invariance}\cite{calabi} Let $\omega \in c_1(H)$ be an extremal metric. Then the isometry group $\Isom_{\nu}(X,\omega) = \Isom(X,\omega) \cap \Aut_{\nu}(X,H)$ of $\omega$ is a maximal compact subgroup of $\Aut_{\nu}(X,H)$.

 \end{theorem}

Thus, for example, cscK metrics are automatically invariant under a maximal compact subgroup of the automorphism group of the manifold. Denote by $T\subset \Isom_{\nu}(X,\omega)$ a maximal torus, with Lie algebra $\mft$ and corresponding space of holomorphy potentials $\bar\mft$. 

\begin{theorem} Let $\omega \in c_1(H)$ be an extremal metric, and denotes $C^{k,\alpha}(X,\R)^T \subset C^{k,\alpha}(X,\R)$ denote the space of $T$-invariant functions on $X$, with $k \geq 4$. Then $$\scL: C^{k,\alpha}(X,\R)^T \to C^{k-4,\alpha}(X,\R)$$ sends real torus invariant functions to real torus invariant functions. Moreover, the operator \begin{align*}&\scL': C^{k,\alpha}(X,\R)^T\times \bar \mft \to C^{k-4,\alpha}(X,\R)^T, \\ & \scL'(\varphi,h) = \scL(\varphi) + h\end{align*} is surjective.
\end{theorem}

\begin{remark}The fact that we linearise at an extremal metric and work with torus invariant functions is crucial here. Indeed, it follows for example from \cite[p. 16]{arezzo-pacard-singer} that provided one works with torus invariant functions, the Lichnerowicz operator is a real operator, which is not true more generally. Slightly more explicitly, if $\omega$ is extremal, then $J\nabla S(\omega)$ is a real holomorphic vector field lying in some compact torus $T$ of automorphisms of $(X,H)$. Then the operator $\scL$ is a real operator on $T$-invariant functions (of course this special case also follows from the fact that the scalar curvature is a real function). Thus, for example, if $\omega$ is a cscK metric, then the operator $\scL$ is real valued without requiring any torus invariance.  The surjectivity is straightforward from ellipticity, and hence surjectivity orthogonal to the kernel, of the operator.\end{remark}

\subsection{Moduli and fibrations} Consider a fibration $\pi: X \to B$ of relative dimension $m$ such that each fibre $X_b = \pi^{-1}(b)$ is smooth for $b \in B$, so that $\pi$ is a holomorphic submersion. We let $n$ be the dimension of $B$, so that $X$ has dimension $m+n$. We endow $X$ with a closed $(1,1)$-form $\omega \in c_1(H)$ such that $\omega$ restricts to a K\"ahler metric on each fibre, which is to say that $\omega$ is a \emph{relatively} K\"ahler metric and $H$ is relatively ample.

Given a $(p,p)$-form $\eta$ on $X$, the fibre integral produces a $(p-m,p-m)$-form $\int_{X/B}\eta$ on $B$ as follows. One first reduces to the local case by a partition of unity argument. Next, using the submersion structure one writes locally over $\pi^{-1}(U)$ with $U \subset B$ $$\eta = \zeta \wedge \pi^*\kappa$$ for some $(p-n,p-n)$-form $\kappa$ on $B$, and then defines $$\int_{X/B}\eta|_{U} = \left(\int_{X/B}\zeta\right)\kappa,$$ where $\int_{X/B}\zeta$ is naturally viewed as a function on $X$. This local construction then globalises to give the desired $(p-m,p-m)$-form $\int_{X/B}\eta$.

The relatively K\"ahler metric $\omega_X$ naturally induces a hermitian metric on the vertical tangent bundle $\scV = \ker d\pi$, which is a holomorphic vector bundle of rank $m$ by the hypothesis that $\pi$ is a holomorphic submersion. We denote the curvature of this metric by $\rho \in c_1(K_{X/B}) = c_1(K_X\otimes(\pi^*K_B^{-1})) = c_1( \wedge^m \scV)$ (we shall return to this construction in more detail in Section \ref{sec:optimal}). On each fibre $X_b$ for $b \in B$, $\rho$ restricts to the Ricci curvature $\Ric\omega_b$. 

\begin{definition}\cite{DN,fujiki-schumacher}Suppose $\omega_X\in c_1(H)$ restricts to a cscK metric $\omega_b$ on every fibre over $b \in B$. We define the \emph{Weil-Petersson metric} to be  $$\alpha =  \frac{S(\omega_b)}{m+1}\int_{X/B} \omega^{m+1} - \int_{X/B}\rho \wedge \omega_X^m.$$  \end{definition}

This is a closed, smooth $(1,1)$-form on $B$. Note that $S(\omega_b)$ is actually independent of $b$ as it is a topological constant. One can show using uniqueness of cscK metrics that this is actually independent of fibrewise cscK metric in the class chosen \cite{DN}. The positivity properties of $\alpha$ on $B$ are precisely related to the moduli theory of the fibres. The precise moduli theory we shall require is the following, which extends the classical work of Fujiki-Schumacher \cite{fujiki-schumacher}. 

\begin{theorem}\cite[Theorem 1.1]{DN} There exists a a complex space $\M$ which is a moduli space of polarised manifolds admitting a cscK metric. More precisely, points of $\M$ are in bijection with cscK polarised manifolds, and a family $\pi: X\to B$ induces a map $q: B \to \M$ compatible with this bijection. \end{theorem}

The link with the Weil-Petersson metric is as follows.

\begin{theorem}\cite[Theorem 1.3]{DN} There is a K\"ahler metric $\Omega_{WP}$ on $\M$ such that if $\pi: (X,H)\to B$ is a family of cscK manifolds with associated Weil-Petersson metric $\alpha$ then  $q^*\Omega_{WP} = \alpha$. \end{theorem}

In summary, from $\pi: (X,H) \to B$ we obtain a moduli map $q: B \to \M$ in such a way that the $(1,1)$-form $\alpha$ is the pullback of a K\"ahler metric on $\M$.

\subsection{Twisted extremal metrics} Extremal metrics are the most natural choice of K\"ahler metric on a K\"ahler manifold, and as we have explained there are deep links between the scalar curvature and the automorphisms of the underlying manifold. Suppose now that $$q: (B,L) \to (M,L_M)$$ is a map between K\"ahler manifolds. More generally, we shall allow $M$ to be a non-compact complex space. It is natural to ask if there is a K\"ahler metric on $X$ whose behaviour reflects the geometry of the map itself. This was achieved in our previous work \cite{fibrations}, and we now explain the consequences which we will use.

\begin{definition} Fix a K\"ahler metric $\Omega \in L_M$. We say that $\omega \in c_1(L)$  is a \begin{enumerate}[(i)] \item \emph{twisted cscK metric} if $S(\omega) - \Lambda_{\omega}q^*\Omega$ is constant;
\item \emph{twisted extremal metric} if $\scD_{\omega}(S(\omega) - \Lambda_{\omega}q^*\Omega),$ which is to say that $S(\omega) - \Lambda_{\omega}q^*\alpha$ is a holomorphy potential on $B$. \end{enumerate}
\end{definition}

Note that only the pullback of $\alpha$ to $X$ itself is involved in the definition. Thus the singularities and compactness of $M$ are irrelevant, provided $q^*\Omega$ is smooth on $B$.

\begin{definition}\label{def:map}\cite[Definition 1]{fibrations} Let $q: (B,L) \to (M,L_M)$ be a map between K\"ahler manifolds. We define $\Aut(q) \subset \Aut(B,L)$ to be the Lie subgroup of automorphisms which fix $q$, in the sense that $$\Aut(q) = \{g \in \Aut(B,L)\ | \  q \circ g = q\}.$$ We denote by $\mfp$ the Lie algebra of $\Aut(p)$, with $\bar\mfp\subset C^{\infty}(X,\C)$ the associated vector space of holomorphy potentials.
\end{definition}

We will only be interested in the \emph{geometric} case of twisted extremal metrics, when the associated holomorphic vector field is an element $\nu \in \mfp$. We then further define $\Aut_{\nu}(p)$ to be the automorphisms of $p$ whose flow commutes with $\nu$.

\begin{theorem}\cite[Corollary 4.2]{fibrations}\label{isometriesoftwistedextremal} Suppose $\omega$ is a twisted extremal metric. Then isometry group $\Isom_{\nu}(q,\omega) = \Isom(B,\omega) \cap \Aut_{\nu}(q)$ is a maximal compact subgroup of $\Aut_{\nu}(q)$.
\end{theorem}

The proof of this relies on the following result regarding the linearisation of the twisted scalar curvature. In the following statement, we denote by $T_q \subset \Isom(q,\omega)$ any compact torus containing the flow of the vector field $\nu$.

\begin{theorem}\cite[Proposition 4.3]{fibrations}\label{fibrations-result} Denote by $\scL_{\Omega}$ the linearisation of the twisted extremal operator \begin{align*} &C^{\infty}(B,\R) \times \bar\mfp \to C^{\infty}(B,\R), \\ &(\varphi,f) \to S(\omega+\ddc\varphi) - \Lambda_{\omega + \ddc \varphi} q^*\Omega - \frac{1}{2}\langle \nabla \varphi,\nabla f  \rangle - f.\end{align*} Suppose that $\omega$ is a twisted extremal metric, with $f = S(\omega) - \Lambda_{\omega}q^*\Omega$. Then $$\scL_{\Omega}(\phi,0) = -\scD_{\omega}^*\scD_{\omega}\phi + \frac{1}{2}\langle\nabla\Lambda_{\omega}q^*\Omega,\nabla \varphi\rangle + \langle \ddc \varphi,q^*\Omega \rangle.$$ Defining $$\scL'_{\Omega}: C^{\infty}(B,\R)^{T} \times \bar\mfp \to  C^{\infty}(B,\R)^{T}$$ by $\scL'_{\Omega}(\varphi,h) = \scL_{\Omega}(\varphi,0) - h$, then $\scL'_{\Omega}$ is well-defined and surjective. 
\end{theorem}

As in the extremal case, the result above is obtained by ellipticity and an identification of the kernel of $\scL_{\Omega}$ with the space of holomorphy potentials for the map to obtain surjectivity, while for $\scL_{\Omega}$ to be a real operator it is again crucial that we linearise at a twisted extremal metric and assume invariance under the flow $\nu$. While \cite[Proposition 4.3]{fibrations} assumes invariance under a maximal compact sbgroup of the automorphism group, from the proof it is clear that it is sufficient to assume invariance under the flow of $\nu$.

We return now to the setting of a fibration $\pi: (X,H) \to (B,L)$ a fibration with all fibres cscK. In this case we obtain a moduli map $q: (B,L)\to (\M, [\Omega_{WP}]).$ 

\begin{lemma}\label{lifting-moduli} Suppose $g\in\Aut(X,H)$ is a lift of an automorphism of $(B,L)$. Then $g$ is the lift of an automorphism of $q$. \end{lemma}

\begin{proof}

For $g \in \Aut(B,L)$ to lift to $X$, we mean that $g_X: (X,H) \to (X,H)$ is induced by the pullback diagram $$\begin{tikzcd}(X,H) \times_B B \arrow[r,"g_X"]\arrow[d]& (X,H)\arrow[d,] \\ B\arrow[r,"g"]& B\end{tikzcd}.$$ For this to be the case, it is necessary that in $\Aut(B,L)$-orbits the fibres are isomorphic, namely $\pi^{-1}(b) \cong\pi^{-1}(g.b)$ for all such $g$. But this precisely states that $g$ is an automorphism of the moduli map $q$.
\end{proof}

This was implicit in \cite{fibrations}. A converse to this was also proved in \cite[Proposition 3.10]{fibrations} in the case that all fibres have discrete automorphism group using some ideas from moduli theory. However, the converse is no longer true in the case that the fibres have continuous automorphisms. For example, consider a projective bundle $\pr(E) \to B$. Then the moduli map $p: B \to \{pt\}$ is trivial, so all automorphisms of $(B,L)$ are automorphisms of $p$. Not all such automorphisms do actually lift, and the additional requirement is that the vector bundle is a  $\Aut(B,L)$-equivariant bundle. Thus we will typically assume later that all automorphisms of $q$ lift to $(X,H)$ in the above sense.

\section{Optimal symplectic connections}\label{sec:optimal}

\subsection{Function spaces and projections}

Let $\pi: (X,H) \to B$ be a fibration, and let $\omega_X \in c_1(H)$ be a relatively K\"ahler metric whose restriction to each fibre $(X_b,\omega_b)$ is cscK. We wish to define a splitting of the space of smooth functions $C^{\infty}(X,\R)$ on $X$ using this K\"ahler metric. 

On each fibre $X_b$, the K\"ahler metric $\omega_b$ defines a Lichnerowicz operator $$\scD_b^*\scD_b= (\db\nabla_b^{1,0})^*\db\nabla_b^{1,0}: C^{\infty}(X_b,\R) \to C^{\infty}(X_b,\R).$$ As $b\in B$ varies, these operators vary smoothly and hence induce an operator $$\scD_{\scV}^*\scD_{\scV}: C^{\infty}(X,\R) \to C^{\infty}(X,\R),$$ defined by the property that $(\scD_{\scV}^*\scD_{\scV}(f))|_{X_b} = \scD_b^*\scD_b(f).$ 

As in Section \ref{sec:preliminaries}, integration over the fibres using the relatively K\"ahler metric $\omega_X$ defines a projection $C^{\infty}(X,\R) \to C^{\infty}(B,\R)$ via $f \to \int_{X/B}f\omega_X^m$, i.e. $$\left(\int_{X/B}f\omega_X^m\right)(b) = \int_{X_b} f|_{X_b} \omega_b^m.$$ Let $C^{\infty}_0(X,\R)$ denote the kernel of this projection, namely the functions which integrate to zero over each fibre. Then $\ker \scD_{\scV}^*\scD_{\scV}(f) \subset C^{\infty}_0(X,\R)$ defines a vector subspace which we denote $$C^{\infty}_E(X) = \ker \scD_{\scV}^*\scD_{\scV} \subset C^{\infty}_0(X,\R).$$ Thus each function $f \in C^{\infty}_E(X)$ restricts to a holomorphy potential with mean value zero on each fibre. 

The K\"ahler metric $\omega_b$ defines an $L^2$-inner product on functions $\phi,\psi \in C^{\infty}_0(X_b)$ on each fibre $X_b$ by $$\langle\phi,\psi\rangle_{b} = \int_X \phi\psi\omega_b^m.$$ The space $\ker\scD_b^*\scD_b\subset C^{\infty}_0(X_b)$ is a finite dimensional vector space by ellipticity of $\scD_b^*\scD_b$. The $L^2$-inner product defines a splitting $$C^{\infty}_0(X_b) \cong \ker \scD_b^*\scD_b \oplus C^{\infty}_R(X_b,\R),$$ by declaring that $C^{\infty}_R(X_b,\R)$ is the $L^2$-orthogonal complement to $\ker \scD_b^*\scD_b$ with respect to the volume form induced by $\omega_b$. This splitting depends only on the inner product, thus since $\omega_b$ and $\scD_b^*\scD_b$ vary smoothly with $b$, the projection defines a splitting $$C^{\infty}_0(X,\R) \cong  C^{\infty}_E(X) \oplus C^{\infty}_R(X,\R),$$ where $\phi \in C^{\infty}_R(X,\R)$ if and only if $\phi|_{X_b} \in C^{\infty}_R(X_b,\R)$ for all $b \in B$.  

In summary, this defines a splitting of the space of functions on $X$ as \begin{equation}\label{eq:function-decomposition}C^{\infty}(X,\R) \cong C^{\infty}(B,\R) \oplus C^{\infty}_E(X,\R) \oplus C_R^{\infty}(X,\R),\end{equation}  where if $\phi\in C^{\infty}(X,\R)$ decomposes as $\phi = \phi_B + \phi_E + \phi_R$, then $\phi_B$ is equal to the pullback of a function on $B$, $\phi_E$ restricts to a holomorphy potential on each fibre, and $\phi_R$ integrates to zero over each fibre and is fibrewise $L^2$-orthonormal to the fibrewise holomorphy potentials. 

It will be convenient to give $C^{\infty}_E(X,\R)$ the structure of the global sections of a smooth vector bundle $E$. One natural way of doing this is as follows. The fibration $X \to B$ admits a structure of a smooth fibre bundle since the fibres are compact \cite[Lemma 17.2]{michor}. This allows one to view $C^{\infty}(X,\R)$ as the global sections of an infinite dimensional vector bundle over $B$, whose fibre over $b \in B$ equals $C^{\infty}(X_b,\R)$. The space $C^{\infty}_0(X,\R)$ is naturally a vector subbundle, for example as the kernel of the bundle morphism to the trivial bundle $\R \times B \to B$, which on global sections corresponds the map $C^{\infty}(X,\R) \to C^{\infty}(B)$ given by the fibre integral. Similarly, the bundle morphism $$\scD_{\scV}^*\scD_{\scV}: C^{\infty}_0(X,\R) \to C^{\infty}_0(X,\R)$$ has kernel of constant finite rank by hypothesis, hence the sections in the kernel are the global sections of a vector bundle. This kernel is then naturally identified with the global sections of $E$, and hence $E$ admits a natural structure of a smooth vector bundle with fibre $E_b \cong  \ker\scD_b^*\scD_b$, the space of holomorphy potentials on $X_b$ with respect to $\omega_b$. Another approach would be to view $C^{\infty}_0(X,\R)$ as the total space of the adjoint bundle of the Hamiltonian frame bundle \cite[Remark 6.4.11]{symplectic-topology}, viewing $\pi: (X,H) \to (B,L)$ as a symplectic fibration, and proceeding in the same way to produce the bundle $E$.

\subsection{Fibrations and curvature}

The vertical tangent bundle of $X$ is defined to be the kernel $\scV = \ker d\pi: TX \to TB$. As $\pi$ is holomorphic, $\scV$ is a holomorphic vector bundle. The fibration structure then induces a short exact sequence of holomorphic vector bundles \begin{equation}\label{eq-ses}0 \to \scV \to TX \to \pi^*TB \to 0.\end{equation} Since $\omega_X$ restricts to a non-degenerate metric on each fibre, one obtains a horizontal vector bundle $\scH$ and a smooth splitting $TX \cong \scV \oplus \scH$ by $$\scH_x = \{ u \in T_xX\ |\ \omega_0(u,v) = 0 \textrm{ for all } v \in \scV\}.$$ This is a standard procedure in symplectic geometry, and in this context the form $\omega_X$ is said to be a \emph{symplectic connection}. 

The splitting $TX \cong \scV \oplus \scH$ induces splittings in the smooth category of any tensor on $X$. For any tensor $\beta$ we shall refer to its components lying purely in $\scV$ and $\scH$ as the \emph{vertical} and \emph{horizontal} components and denote them by $\beta_{\scV}$ and $\beta_{\scH}$ respectively. 

The relatively K\"ahler metric $\omega_X$, viewed as a symplectic connection, induces an Ehresmann connection on the fibre bundle $\pi: X \to B$. For a vector field $v$ on $B$, we denote by $v^{\#}$ its horizontal lift using the connection, so that for all $x \in X$, $v^{\#}_x \in \scH_x$.

\begin{definition}\cite[p. 277]{symplectic-topology} The \emph{symplectic curvature} $F_{\scH}$ of $\omega_X$ is a $2$-form $F_{\scH} \in \Omega^2(B,Ham(\V))$  with values in Hamiltonian vector fields on the fibres, defined by $$F_{\scH}(v_1,v_2) = [v_1^{\#},v_2^{\#}]^{vert} = [v_1^{\#},v_2^{\#}] - [v_1,v_2]^{\#}.$$
 \end{definition}

The symplectic curvature is simply the curvature of the associated Ehresmann connection.

Let $\mu^*: Ham(\V) \to C^{\infty}_0(X)$ be the fibrewise comoment map, which sends a Hamiltonian vertical vector field on $X$ to its fibrewise mean value zero Hamiltonian using $\omega_b$. The map $\mu^*$ extends to a map on other tensors, and so one can view $\mu^*(F_{\scH})$ as a two-form on $B$ with values in fibrewise Hamiltonian functions. In the same manner, one can further consider $\mu^*(F_{\scH})$ as a two-form on $X$ with values in fibrewise Hamiltonian functions, by pulling back the form from $B$.

\begin{lemma}\cite[Equation (1.12)]{symplectic-fibrations-book}\label{minimal-coupling} There is a two-form $\beta$ on $B$ such that $$(\omega_X)_{\scH} = \mu^*(F_{\scH}) + \pi^*\beta.$$\end{lemma}

This is known as \emph{minimal coupling} in the symplectic geometry literature. To understand the horizontal component $(\omega_X)_{\scH}$ more explicitly, we proceed as follows. Note that  $\mu^*(F_{\scH})$ has direct image zero: $\mu^*(F_{\scH})$ is of the form $\mu^*(F_{\scH}) = f\pi^*\nu$, where $\int_{X/B} f \omega_X^m=0$ by definition of the comoment map, hence $$\int_{X/B} \mu^*(F_{\scH})\wedge \omega_X^m = \int_{X/B} f\pi^*\nu \wedge \omega_X^m = \nu \int_{X/B}f\omega_X^m = 0.$$ Thus in the Lemma above, setting $V =\int_{X_b}\omega_b^m$ to be the volume of any fibre, the form $\beta$ is given as $$\beta =\int_{X/B}(\omega_X)_{\scH}\wedge\frac{\omega_X^m}{V}.$$ 

The final objects involved which we shall make use of are the horizontal and vertical contractions, and vertical and horizontal Laplacians. The horizontal contraction and horizontal Laplacian will involve a K\"ahler metric  on $B$, which we denote $\omega_B \in c_1(L)$.

\begin{definition} Let $\beta$ be a two-form on $X$. We define 
\begin{enumerate}[(i)] \item the \emph{vertical contraction} to be $$\Lambda_{\scV}\beta = m\frac{\beta_{\scV} \wedge \omega_X^{m-1}}{\omega_X^m},$$ where  the quotient is taken in $\det \scV^*$.
\item the \emph{horizontal contraction} to be $$\Lambda_{\omega_B}\beta = n\frac{\beta_{\scH}\wedge \omega_B^{n-1}}{\omega_B^n},$$ where the quotient is taken in $\det \scH^*$.
\end{enumerate}
\end{definition}

\begin{remark} A two-form $\beta$ as above restricts to a form $\beta|_{X_b}$ on each fibre, and the vertical contraction is simply the fibrewise contraction of $\beta$ with respect to the fibre metric $\omega_b$. Similarly for forms pulled back from $B$, the horizontal contraction is just the pullback of the usual contraction computed on $B$. \end{remark}

\begin{definition} Let $f$ be a function on $X$. We define the \emph{vertical Laplacian} to be  $$\Delta_{\scV}f = \Lambda_{\scV} (\ddc f).$$
\end{definition}

The vertical and horizontal contractions, the vertical Laplacian and their variants play an important role in the previous work on extremal metrics on holomorphic submersions \cite{hong1,fine1, bronnle, lu-seyyedali,fibrations}. 

In addition to the symplectic curvature itself, the equation we will be interested in also involves a relative version of the Ricci curvature. As $\omega_X$ is positive on each fibre, it induces a hermitian metric $g$ on $\scV$ and hence a hermitian metric $\wedge^m g$ on $\wedge^m \scV$. Denote by $\rho$ the curvature of this metric, so that $\rho \in c_1(\wedge^m\scV) = c_1(K_{X/B}).$ We emphasise that $\rho$ depends only on $\omega_X$, and hence is independent of $k$. While its construction uses only $(\omega_X)_{\scV}$ and not its horizontal component, in general $\rho$ itself, being the curvature of a hermitian metric on the line bundle $\wedge^m \scV$, will have a horizontal component. Remark also that $\rho$ restricts to the Ricci curvature of the restriction of $\omega_X$ on any fibre, since $\wedge^m\scV \cong -K_{X/B}$ restricts to $-K_{X_b}$ on any fibre and curvature commutes with pullback.

\subsection{The optimal symplectic connection equation} 

We continue following the notation introduced above, which we briefly summarise. We begin with a holomorphic submersion $\pi: (X,H) \to (B,L)$ with $\omega_X \in c_1(H)$ cscK on each fibre and $\omega_B \in c_1(L)$ K\"ahler. The symplectic curvature of $\omega_X$ is denoted $\mu^*(F_{\scH})$. The metric induced by $\omega_X$ on the top exterior power of the vertical tangent bundle $\wedge^m\scV$ has curvature $\rho$, which has a horizontal component $\rho_{\scH}$. $\omega_X$ also defines a fibrewise, vertical Laplacian operator $\Delta_{\scV}$. Finally, the K\"ahler metric $\omega_B$ on $B$ induces a contraction operator $\Lambda_{\omega_B}$ on horizontal forms. 

\begin{definition} We say that $\omega_X$ is an \emph{optimal symplectic connection} if $$p(\Delta_{\scV}(\Lambda_{\omega_B} \mu^*(F_{\scH})) + \Lambda_{\omega_B} \rho_{\scH})=0.$$ \end{definition}

As described in the Introduction, we think of this as a canonical choice of fibrewise cscK metric. For this to be reasonable, the equation should only depend on the behaviour of $\omega_X$ restricted to each fibre. To make this precise, we first show that solutions remain solutions upon the addition of a pullback from $B$.

\begin{lemma} If $\omega_X$ is an optimal symplectic connection, then so is $\omega_X + \pi^*{\nu}$ for any closed $(1,1)$-form $\nu$ on $B$. \end{lemma}

\begin{proof} This follows simply because all objects involved are unchanged when one adds a form from $B$. Explicitly, the symplectic curvature $\mu^*(F_{\scH})$ is unchanged by the addition of a form due to the minimal couple equation of Lemma \ref{minimal-coupling}. The construction of $\rho$ only involves the restriction of the $\omega_X$ to the fibres $X_b$ so is also unchanged. The horizontal component $\rho_{\scH}$ is then unchanged since pulling back a form from $B$ does not change the vertical horizontal decomposition induced by $\omega_X$ thought of as a symplectic connection. Finally, since the vertical Laplacian and projection operators also only involve the restriction to a given fibre, they too are unchanged. \end{proof}

A natural question our work raises is, given some initial fibrewise cscK metric $\omega_X$, how does one find an optimal symplectic connection? Here we make some remarks on the steps needed to at least phrase this as a PDE on $E$. If $\omega$ is a cscK metric in some fixed class, then any other cscK metric is of the form $g^* \omega$ for some automorphism $g$, by Theorem \ref{thm:uniqueness}. Moreover, $g$ can be written as the time one flow of a vector field $v_g$ in the Lie algebra of automorphism group of the manifold. In our case, when we change the fibrewise holomorphy potentials this means any other cscK metric on a given fibre is determined from by some $v_g \in \Lie(\Aut(X_b,H_b)).$ Moreover, $v_g$ has holomorphy potential $h_b\in C^{\infty}_0(X_b)$, and so we can determine any other smooth fibrewise metric from some function in $h \in C^{\infty}_E (X)$, i.e. from some global section of $E$. 

By the $\bar \partial \partial $-Lemma, if we have another fibrewise metric, it can be written in the form $\omega_{\phi} = \omega_X + \ddc \phi$, for some $\phi : X \to \R.$ From what we have said above, if this metric is also fibrewise cscK, then $\phi$ is uniquely determined by a function $h \in C^{\infty}_E (X)$. However, it is not true that $\phi = h.$ Explicitly, $\phi_{|X_b} = \int_0^1 f_t^*( h_b) dt,$ where $f_t$ is the time $t$ flow on $X_b$ produced from $h_b$ \cite[Example 4.26]{szekelyhidi-book}. Now, if we are looking to solve the optimal symplectic connection, we wish to solve this \textit{with respect to $\omega_{\phi}$}, not $\omega_X$. A crucial point is that we are thinking of $E$ is a fixed bundle of fibrewise vector fields, but its realisation in terms of functions on $X$ depends on the choice of fibrewise cscK metric. If we denote by $E_{\phi}$ the new realisation of this bundle in terms of $\omega_{\phi}$ fibrewise holomorphy potentials, we have an isomorphism $P : E \to E_{\phi}$, given by the change in holomorphy potentials. To get a PDE on our initial bundle $E$, we then need to apply the inverse of $P$. This seems like the natural strategy to try to setup and solve the optimal symplectic connection in general, starting from some arbitrary fibrewise cscK metric. 

Of course, for this strategy to be sensible, one needs to first be able to pick a reference fibrewise cscK metric. It is a non-trivial consequence of the deformation theory of cscK manifolds, due to Br\"onnle and Sz\"ekelyhidi independently \cite{TB,GS}, that fibrewise cscK metrics exist, as we now establish. This was already noted by Li-Wang-Xu  in the setting of K\"ahler-Einstein fibrations \cite[Section 4]{lwx2} .

\begin{lemma} Let $(X,H) \to (B,L)$ be a fibration such that $(X_b,H_b)$ admits a cscK metric for all $b \in B$. Then there exists an $\omega_X \in c_1(H)$  which is relatively cscK. \end{lemma}

\begin{proof} Fix a distinguished point $b_0 \in B$. Then over a sufficiently small neighbourhood $U_0$ of $b_0$ in $B$, there is a relatively K\"ahler metric $\omega$ such that $\omega|_{\pi^{-1}(U_0)}$ is cscK, by Br\"onnle and Sz\"ekelyhidi's deformation theory of cscK manifolds \cite{TB,GS}. Slightly shrinking $U$, by the same deformation theory, there is a $\phi \in C^{\infty}(U,\R)$ such that $\omega+\ddb\phi$ is cscK on \emph{each} fibre over $U$. Indeed, Br\"onnle and Sz\"ekelyhidi's work implies that the cscK metric can be taken to vary smoothly in families locally, which essentially just uses that the zero set of the moment map they consider is connected. Cover $B$ by finitely many such charts $U_j$ each with relatively cscK metric $\omega_j$. While it is not true that $\omega_j = \omega_{j'}$ on $\pi^{-1}(U_j \cap U_{j'})$, if we let $$\omega_j - \omega_{j'} = \ddb \phi_j,$$ then there is a smooth function $\phi_{j,t}$ such that $\omega_j - \ddb \phi_{j,t}$ is a cscK metric for all $t \in [0,1]$ with $\phi_{j,0}=0$ and $\phi_{j,1}=\phi_j$. Indeed, on each fibre, by uniqueness of cscK metrics $\omega_{j'}$ is the pullback of $\omega_{j'}$ through the time one flow of a real holomorphic vector field, so pulling back $\omega_{j'}$ through the time $t$ flow produces such a smooth family of fibrewise K\"ahler potentials. One can then use a cutoff function $\rho_{jj'}$ on $U_j \cap U_{j'}$ and modify $\omega_j$ by adding $\ddb\pi^*(\rho_{jj'})\phi_j$ to glue the metrics. Since $\pi^*\rho_{jj'}$ on $X$ is the pullback of a function on $U_j \cap U_{j'}$, the resulting form will still be relatively K\"ahler. Iterating this procedure over all intersections will produce a relatively cscK metric, as required.\end{proof}

Just as the cscK equation reduces to the K\"ahler-Einstein equation under the appropriate topological hypothesis, the optimal symplectic connection equation simplifies under an appropriate topological condition. We begin with the following elementary Lemma.

\begin{lemma}\label{restriction-of-forms} Suppose $\omega_X, \tilde \omega_X$ are arbitrary closed $(1,1)$-forms on $X$ such that for all $b \in B$ $$\omega_X|_{X_b} = \tilde\omega_X|_{X_b}.$$ Then $$\omega_X = \tilde \omega_X + \pi^*{\nu}$$ for some closed $(1,1)$-form $\nu$ on $B$.
\end{lemma}

\begin{proof} 

It follows from the hypotheses that $$[\omega_X] = [\tilde \omega_X] + \pi^*\beta$$ for some $\beta \in H^{1,1}(X,\R)$. Picking an arbitrary form $\nu_{\beta} \in \beta$ and replacing $\tilde \omega_X$ with $\tilde\omega_X + \nu_{\beta}$, it is enough to prove the Lemma in the case  $[\omega_X] = [\tilde \omega_X]$. Then by the $\partial\bar\partial$-Lemma, we have $\omega_X  = \tilde \omega_X + \ddb \phi$ for some $\phi \in C^{\infty}(X)$. Write $\phi = \phi_B + \phi_X$ with $\phi_B \in C^{\infty}(B)$ and $\phi_X \in C^{\infty}_0(X)$, where the decomposition is induced by any reference relatively K\"ahler metric. Then the result follows immediately since $\phi_X = 0$, as $\omega_X|_{X_b} = \tilde\omega_X|_{X_b}$ and $\phi_X$ integrates to zero over the fibres.\end{proof}

\begin{proposition}

Suppose $\omega_X \in c_1(H)$ is a relatively cscK metric such that $$\lambda[\omega_b] =  c_1(X_b)$$ for all $b \in B$ and for some $\lambda \in \R$ independent of $b$. Then $\omega_X$ is an optimal symplectic connection if and only if $$p(\Delta_{\scV}(\Lambda_{\omega_B} \mu^*(F_{\scH})) + \Lambda_{\omega_B} \mu^*(F_{\scH}))=0.$$
\end{proposition}

\begin{proof} This is an immediate consequence of what we have already proved. Firstly, since $\omega_b$ is cscK, it is automatically K\"ahler-Einstein by Remark \ref{KEorcscK} and hence $\lambda\omega_b = \Ric \omega_b$ by the topological hypothesis of the Theorem. It follows from Lemma \ref{restriction-of-forms} that $\omega_X - \rho = \pi^*\nu$ for some closed $(1,1)$-form $\nu$ on $\beta$, so $\mu^*(F_{\scH}) = \rho_{\scH}$, as required.
 \end{proof}
 
The only interesting case of the above is when $\lambda>0$, so that the fibres $X_b$ are all Fano manifolds. Indeed, in the other two cases the fibres are either Calabi-Yau or of general type, and hence $\Aut(X_b,H_b)$ is discrete for all $b \in B$. In these cases, a fibrewise cscK metric is uniquely determined, and the projection operator $p$ is the zero map, so the optimal symplectic connection condition is trivial.

\begin{definition} The \emph{relative automorphism group} $\Aut(X/B,H)$ is defined to be $$\{g \in \Aut(X,H): \ \pi\circ g = \pi \}.$$\end{definition}

\begin{proposition}\label{prop:invariance}

Suppose $g \in \Aut(X/B,H)$, and let $\omega_X$ be an optimal symplectic connection. Then $g^*\omega_X$ is also an optimal symplectic connection.

\end{proposition}

\begin{proof}

We show directly that that each of the quantities in the optimal symplectic connection equation are a pullback of the quantities using the pullback metric. To simplify notation, we shall add a $g$ subscript to all quantities measured using $g$, so that $F_{\scH_g}$ is the symplectic curvature of the symplectic connection $g^*\omega_X$. Thus what we wish to prove is that $$g^*(p(\Delta_{\scV}(\Lambda_{\omega_B} \mu^*(F_{\scH})) + \Lambda_{\omega_B} \rho_{\scH})) = p_g(\Delta_{\scV_g}(\Lambda_{\omega_B} \mu^*(F_{\scH_g})) + \Lambda_{\omega_B} \rho_{\scH_g}),$$ from which it would follow that the right hand side vanishes if $\omega_X$ is an optimal symplectic connection.

First of all note that, for a function $f\in C^{\infty}(X)$, we have $g^*(p(f)) = p_g(g^*f)$. Indeed, fibrewise holomorphy potentials with respect to $g^*\omega_X$ are just the pullback via $g$ of the holomorphy potentials for $\omega_X$. Similarly, the Laplacian operator satisfies $\Delta_{\scV_g} (g^*f) = g^*(\Delta_{\scV}(f)).$ Since $g \in \Aut(X/B,H)$, the horizontal contraction is unchanged by $g$. The symplectic curvature is again just the pullback, since the curvature of a pullback connection is the pullback of the curvature (this can also be seen through Lemma \ref{minimal-coupling}). The fact that $\rho_{\scH_g}$ is the pullback of $\rho_{\scH}$ follows from a similar argument, and since $\rho$ is defined via taking a curvature, also uses that $g$ is holomorphic.
\end{proof}

The invariance of the equation under the action of $\Aut(X/B,H)$ immediately produces the following.

\begin{lemma} Suppose $\phi$ is a fibrewise holomorphy potential whose flow lies in $\Aut(X/B,H)$. Then $\phi $ is in the kernel of the linearisation of the optimal symplectic curvature operator at a solution. \end{lemma}

\begin{proof} This is an immediate consequence of the invariance of the equation, which means that the derivative along the direction induced by $\phi$ of the equation is zero. \end{proof}

We return to the linearisation of the optimal symplectic curvature equation in Section \ref{sec:approx}.

We will later use the following result regarding the group $\Aut(X/B,H)$, which generalises part of \cite[Proposition 3.10]{fibrations}. We use the notation $q: B \to \M$ for the moduli map, with $\Aut(q)$ defined in Definition \ref{def:map}.

\begin{proposition}\label{automsofX}Suppose that all automorphisms of $q$ lift to $X$.  Then there is a short exact sequence $$0 \to \Lie(\Aut(X/B,H) \to \Lie(\Aut(X,H)) \to \Lie(\Aut(q)) \to 0.$$ \end{proposition}

\begin{proof} We first prove an analogous statement for $H^0(X,TX^{1,0})$, which consists of global holomorphic vector fields which may not have a zero. We use the long exact sequence $$0 \to H^0(X,\scV)  \overset{\alpha}\to H^0(X,TX^{1,0})  \overset{\beta}\to H^0(X,\pi^*TB) \to\hdots $$ associated to the short exact sequence of equation \eqref{eq-ses}. 

Let $u \in H^0(X,TX^{1,0})$. By Lemma \ref{lifting-moduli}, $\beta(u)$ must be a vector field whose flow fixes the moduli map $q$. If we assume  all automorphisms of $q$ lift to $X$, then there is a lift of $\beta(u)$, say $u_q$, which satisfies $\beta(u_q) = \beta(u)$. As $\beta(u_q - u) = 0$, it follows that there is a $v \in H^0(X,\scV)$ with $\alpha(v) =u- u_q$. Hence we can write $u = \alpha(v + v') + u_q$, where $u_q$ is a lift of a holomorphic vector field on $B$ whose flow lifts to $X$. 

We now prove the refined result we are interested in, namely assuming that the vector field $u \in H^0(X,TX^{1,0})$ has a zero, we wish to write it as a sum of elements of $\Lie(\Aut(q))$ (or rather a lift of such an element) and $\Lie(\Aut(X/B),H)$. We first prove this for rational elements of $\Lie\Aut(X,H)$. Then the flow of $u$ generates a $\C^*$-action, since rational vector fields generate $\C^*$-actions. Following the proof above, if $u$ generates a $\C^*$-action, so must $\beta(u)$, and hence so must the lift $u_q$ of $\beta(u)$. As above, since $\beta(u_q - u) = 0$, it follows that there is a $v \in H^0(X,\scV)$ with $\alpha(v) =u- u_q$. Then since $u_q$ is rational, so is $u- u_q$, and hence so is $v$. Thus the flow of $v$ generates a $\C^*$-action. It is then clear that the flow of $v$ has a fixed point on \emph{each} fibre $X_b$, which means that it is an element of $\Lie(\Aut(X/B),H)$, while for the same reason $u_q$ is an element of $ \Lie(\Aut(q))$ (or again rather a lift of such an element). 

To obtain the conclusion for irrational vector fields, write $u = \alpha(v) + u_q$ where wish to show that $v$ has a zero on every fibre and $u_q$ has a zero. Certainly $u_q$ must have a zero just as above. Approximate $u$ by rational elements of the Lie algebra $u_j = \alpha(v_j) + u_{q,j}$, where by the previous paragraph we know that the desired conclusion holds for the $u_j$. We may assume that the $v_j$ converge to $v$, for example by first approximating $v$.. Since the $v_j$ generate a $\C^*$-action, they have a fixed point on \emph{each} fibre, which must vary continuously in $j$. Thus $v$ itself has a zero on each fibre, as required.

Slightly more explicitly, what we have shown is that the map $$\alpha: \Lie(\Aut(X/B,H) \to \Lie(\Aut(X,H))$$ is injective, via the hypothesis that all automorphisms of $q$ lift to $(X,H)$ we have shown that $\Lie(\Aut(X,H)) \to \Lie(\Aut(q))$ is surjective, and we finally we have shown that any element of $\Lie(\Aut(X,H))$ can be written as a sum of elements of  $\Lie(\Aut(q))$ and $\Lie(\Aut(X/B,H)$. This precisely says that the sequence $$0 \to \Lie(\Aut(X/B,H) \to \Lie(\Aut(X,H)) \to \Lie(\Aut(q)) \to 0$$  is exact. 
\end{proof}

There is no canonical way to split the produced short exact sequence in general.

\subsection{Extremal symplectic connections}\label{sec:extremal-symplectic} Just as extremal K\"ahler metrics generalise constant scalar curvature K\"ahler metrics by asking that the scalar curvature is a holomorphy potential, there is an analogous generalisation of optimal symplectic connections to what we call extremal symplectic connections. We will return to these connections in more detail in Section \ref{sec:general-case}; here we briefly give the relevant definition. As before, we assume that $\omega_X$ is cscK on each fibre.

\begin{definition}\label{def:extremalsymplectic} We say that $\omega_X$ is an \emph{extremal symplectic connection} if the function $$p(\Delta_{\scV}(\Lambda_{\omega_B} \mu^*(F_{\scH})) + \Lambda_{\omega_B} \rho_{\scH})$$  is a global holomorphy potential on $X$ with respect to $k{\omega_B} + \omega_X$ for all $k$. \end{definition} 

In particular, an extremal symplectic connection has an associated holomorphic vector field $\nu_E$, whose flow lies in $\Aut((X/B),H)$. 

The definition does not depend on choice of $k$, which we have implicitly taken large enough so that the form $k{\omega_B} + \omega_X$ is K\"ahler, as we show in Section \ref{sec:general-case}. There we will also provide an equivalent definition, which states that $\omega_X$ is an extremal symplectic connection if and only if $$\scR\circ p(\Delta_{\scV}(\Lambda_{\omega_B} \mu^*(F_{\scH})) + \Lambda_{\omega_B} \rho_{\scH})=0,$$ where $\scR$ is a linear differential operator which we will later define, and which plays the analogous role to that which the operator $\scD$ plays for extremal metrics.

\subsection{Projective bundles and coadjoint orbits}\label{coadjoint}

We now explain how the optimal symplectic equation reduces to the Hermite-Einstein condition on projective bundles and certain other fibrations constructed via principal bundles. We begin with the case of projective bundles.

\begin{lemma}\cite[Remark 6.13]{tian-book} Suppose $(X,\omega)$ is a K\"ahler-Einstein Fano manifold with $[\omega] = c_1(X)$. Then the eigenspace of the Laplacian with eigenvalue one consists precisely of the holomorphy potentials of average zero.\end{lemma}

\begin{proof} This is well known, and follows from a direct computation. A quick proof from the perspective we are taking in the present work is to use that holomorphy potentials are precisely the zeros of the Lichnerowicz operator, which from Equation \eqref{eqn:lich} on Fano K\"ahler-Einstein manifolds takes the form $$\D^*\D\phi = \Delta^2\phi - \Delta\phi.$$ Thus $ \D^*\D\phi = 0$ if and only if $\Delta((\Delta - 1)(\phi)) = 0$, giving the result since there are no non-constant harmonic functions on a compact manifold. 
\end{proof}

The proof also demonstrates that it is rather unlikely that holomorphy potentials give eigenfunctions of the Laplacian away from the Fano K\"ahler-Einstein situation. 

\begin{proposition} Suppose $(\pr(E),\scO(1)) \to (B,L)$ is a projective bundle. Then $\omega_X \in \scO(1)$ is an optimal symplectic connection if and only if $\omega_X$ is induced from a Hermite-Einstein metric on $E$.\end{proposition}

\begin{proof}On each, the choice $\omega_b =( \omega_X)|_{X_b}$ of Fubini-Study metric corresponds to a choice of basis of basis of $E_b$, so induces a hermitian metric $h$. If $\omega_h \in c_1(\scO(1))$ is the corresponding form, induced by considering $\scO(1)$ as a subbundle of $\pi^*E \to \pr(E)$ and taking the curvature of the restriction of the pullback metric, then $(\omega_h)|_{X_b} = ( \omega_X)|_{X_b}$ for all $b \in B$ and so by Lemma \ref{restriction-of-forms} the two forms must be equal. If $\mu^*(F_{\scH})$ denotes the curvature, then from the construction via the hermitian metric one sees that $ \mu^*(F_{\scH})$ is actually (the pullback via $\pi$ of) a two-form on $B$ with values in fibrewise \emph{holomorphy potentials} (rather than simply Hamiltonians). Thus $\Lambda_{\omega_B} \mu^*(F_{\scH})$ restricts to a holomorphy potential on each fibre $\pr(E)_b$, so $$\Delta_{\scV} \Lambda_{\omega_B} \mu^*(F_{\scH}) = \Lambda_{\omega_B} \mu^*(F_{\scH}).$$ The projection operator $p$ is the identity on holomorphy potentials (with fibrewise mean value zero as usual), hence $\omega_X$ is an optimal symplectic connection if and only if $$p(\Delta_{\scV}(\Lambda_{\omega_B} \mu^*(F_{\scH})) + \Lambda_{\omega_B} \mu^*(F_{\scH}))= \mu^*(\Lambda_{\omega_B}F_{\scH})=0.$$ Since $\mu^*$ is the fibrewise comoment map with mean value zero, this simply asks that $$\Lambda_{\omega_B}F_{\scH} = \lambda \Id,$$ where $\lambda$ is the appropriate topological constant. But this is simply the Hermite-Einstein equation for the hermitian metric $h$. 
\end{proof}

\begin{remark} \begin{enumerate}[(i)] \item  Of course, the proof also implies that for any line bundle $H$ on $\pr(E)$ such that $H|_{\pr(E)_b} = \scO(l)$ for all $b \in B$ and for some fixed integer $l$, the existence of an optimal symplectic connection is equivalent to the existence of a Hermite-Einstein metric. \item If $E_1\oplus\hdots\oplus E_l$ is a direct sum of stable vector bundles, then the Hermite-Einstein metrics on each factor induce a hermitian metric on the bundle itself. If the slopes of the bundles are distinct, this will not be a Hermite-Einstein metric, but rather will induce an extremal symplectic connection. In this way, Br\"onnle's work concerning extremal metrics on the projectivisation of such bundles is a special case of our results on extremal symplectic connections \cite{bronnle}.
\end{enumerate}
\end{remark}

The situation is similar when the fibres of $(X,H) \to (B,L)$ are integral coadjoint orbits. Br\"onnle speculated that there may be construction of extremal metrics in this situation  \cite[Section 2.3]{bronnle}, here we explain how this follows from our main result. We briefly recall the basic setup, referring to \cite[Section 2.2]{donaldson-notes}, \cite[Section 2.3]{bronnle} and \cite[Chapter 8]{besse} for further details.

Let $G$ be a connected compact Lie group with complexification $G^{\C}$. The group $G$ acts naturally on the dual of its Lie algebra $\mfg^*$ via the coadjoint action. If $\xi \in \mfg^*$, denoting by $K = \Stab(\xi)$ the stabiliser of $F$, $G/K$ can be considered as a submanifold of $\mfg^*$. Then $G/K$ is actually a compact Fano manifold, and the $G$-action on $G/K$ is transitive and holomorphic. The Fano manifold $G/K$ admits a natural symplectic form which is in fact K\"ahler-Einstein. The element $\xi \in \mfg^*$ induces a Lie algebra morphism $\mfg \to \R$, which restricts to a Lie algebra homomorphism $\mfk \to \R$, where we have denoted by $\mfk$ the Lie algebra of $K$. If $\xi: \mfk \to \R$ is induced from an action of $H$ on $S^1$, then $G/K$ is said to be an \emph{integral coadjoint orbit}.  In this case, one naturally obtains holomorphic line bundle over $G/K$ which admits a K\"ahler-Einstein metric and an action of $G$, hence a natural $G$-invariant hermitian metric.

Now suppose $P \to (B,L)$ is a principal $G$-bundle, and let $P^{\C}$ be the associated principal $G^{\C}$-bundle. Since the $G$-action on $G/K$ is holomorphic, one obtains a holomorphic fibre bundle $X \to (B,L)$ associated to $P \to (B,L)$. One similarly obtains a holomorphic line bundle $H$ on $X$ as the associated bundle to the $G$-action on the holomorphic line bundle on $G/K$. 

Let $\nabla$ be a connection on $P \to (B,L)$. Then $\nabla$ uniquely induces a connection on $P^{\C}$ which is unitary and compatible with the complex  tructure in the usual sense of principal bundles  \cite[p. 220]{biswas}. This, together with the hermitian metric inducing the K\"ahler-Einstein metric on $G/K$, induces a connection on $H$ through the $G$-action \cite[Section 2.3]{bronnle}\cite[Remark 2.3]{fine-panov}. A result of Fine-Panov implies that, if $\omega_X$ is the curvature of this metric, then $(\omega_X)_{\scH} = \mu^*(F^{\nabla}),$ while by construction $\omega_X$ restricts to a K\"ahler-Einstein metric on each fibre \cite[Remark 2.3]{fine-panov}. 

The crucial point is that, since the action of $G$ on $G/K$ is holomorphic, the symplectic curvature $F_{\scH}$ of this connection is automatically a two-form on $B$ with values in fibrewise \emph{holomorphic} vector fields \cite[Section 11.9]{michor}. Thus $\mu^*(F^{\nabla})$ is automatically a fibrewise holomorphy potential, and since $G/H$ is Fano, just as for projective bundles we have $$p(\Delta_{\scV}(\Lambda_{\omega_B} \mu^*(F^{\nabla})) + \Lambda_{\omega_B} \mu^*(F^{\nabla}))= \mu^*(\Lambda_{\omega_B}F^{\nabla}),$$ which is zero if and only if the connection $\nabla$ defines a Hermite-Einstein connection, again in the usual sense of principal bundles \cite[Definition 3.2]{biswas}. In summary, the same argument as in the case of projective bundles gives the following:

\begin{corollary} Suppose $(X,H) \to (B,L)$ has fibres which are integral coadjoint orbits, induced from a principal $G^{\C}$-bundle $P^{\C} \to B$ in the manner described above. If $P^{\C}$ admits a Hermite-Einstein connection, then $(X,H)$ admits an optimal symplectic connection.\end{corollary}

In particular, from our main results one can construct extremal metrics on the total space of such fibrations. This answers a question of Br\"onnle  \cite[Section 2.3]{bronnle}.

\begin{remark}If one fixes a polarised manifold $(Y,L_Y)$ with automorphism group $G^{\C} = \Aut(Y,L_Y) $, then holomorphic principal $G^{\C}$-bundles are naturally in correspondence with fibrations with fibre $(Y,L_Y)$. However, even for general Fano fibrations, the optimal symplectic curvature equation does not reduce to a Hermite-Einstein type condition. The issue is that, while the projection operator $p$ allows one to induce a connection on such a principal bundle $P^{\C}$ from a symplectic connection $\omega_X$, it is not true that the projection of the symplectic curvature of $\omega_X$ is equal to the curvature of the induced connection on  $P^{\C}$. 

\end{remark}

\section{The approximate solutions}\label{sec:approx}

We will now construct approximate solutions to any desired order to the extremal equation on the total space of the fibration. 

\subsection{Expansion of the scalar curvature}

Let $\omega_X\in c_1(H)$ be a relatively K\"ahler metric, and let $\omega_B \in c_1(L)$ be K\"ahler. We denote $\omega_k =k\omega_B  + \omega_X $. The goal of this section is to calculate $S(\omega_k)$, as a function of $k$, to leading two orders in $k\gg 0$. This will involve an expansion of the Ricci curvature, and also firstly an expansion of the Laplcian operator.

\begin{remark}\label{norminuse} In the various expansions established in present section, the estimates are \textit{pointwise}. Thus for example $$S(\omega_k) = f_0 + f_1k^{-1} + O(k^{-2})$$ means that for each $x \in X$, there is a $c\geq 0$ such that $$|S(\omega_k)(x) - f_0(x) - k^{-1}f_1(x)| \leq ck^{-2}$$ for all $k \gg 0$. In Section \ref{sec:implicit}, we will improve our estimates to global estimates in the $C^l$-norm, using a patching argument of Fine \cite[Section 5]{fine1}. \end{remark}

\begin{lemma}The contraction operator $\Lambda_{\omega_k}$ satisfies $$\Lambda_{\omega_k}\beta = \Lambda_{\scV}\beta + k^{-1}\Lambda_{\omega_B}\beta + O(k^{-2}).$$
\end{lemma}

\begin{proof} The contraction $\Lambda_{\omega_k}$ is defined to be $$ \Lambda_{\omega_k}\beta = (m+n)\frac{\beta\wedge(k\omega_B + \omega_X)^{m+n-1}}{(k\omega_B+\omega_X)^{m+n}}.$$ Note that $\omega_k$ has no mixed term under the vertical horizontal decomposition of tensors on $X$, since certainly $\omega_B$ does not and $\omega_X$ is used to define the decomposition. Write $\beta = \beta_{\scV} + \beta_{mixed} + \beta_{\scH}$ for the decomposition of $\beta$ into vertical, mixed and horizontal terms. One then expands to obtain, using that $\omega_k$ has no mixed term \begin{align*}\Lambda_{\omega_k}\beta &= m\frac{\beta_{\scV}\wedge(\omega_X)_{\scV}^{m-1}\wedge((\omega_X)_{\scH}+k\omega_B)^n}{(\omega_X)^m_{\scV}\wedge((\omega_X)_{\scH}+k\omega_B)^n} + n\frac{\beta_{\scH}\wedge(\omega_X)_{\scV}^{m}\wedge((\omega_X)_{\scH}+k\omega_B)^{n-1}}{(\omega_X)^m_{\scV}\wedge((\omega_X)_{\scH}+k\omega_B)^n}, \\ &= \Lambda_{\scV}\beta + k^{-1}\Lambda_{\omega_B}\beta + O(k^{-2}),\end{align*} as required.\end{proof}

\begin{corollary} We have $$\Delta_k = \Delta_{\scV} + k^{-1}\Delta_{\scH} + O(k^{-2}),$$ where $\Delta_k$ is the Laplacian determined by $\omega_k$. \end{corollary}

Thus in general the leading order term is the fibrewise Laplacian of $f$, and if $f$ is pulled back from $B$, the leading order term is the pullback of the Laplacian from $B$.

The short exact sequence $$ 0 \to \scV \to TX \to \pi^*TB \to 0$$ induces an isomorphism $$\wedge^nTX \cong \wedge^m\scV \otimes \pi^*\wedge^b TB.$$ The K\"ahler metric $\omega_k$ induces a hermitian metric on $TX$, and hence on $\wedge^nTX$, whose curvature is $\Ric(\omega_k)$. From the previous short exact sequence, $\Ric(\omega_k)$ can be written as a sum of the curvatures of the metrics on $ \wedge^m\scV$ and $\wedge^b\scH \cong  \pi^*\wedge^b TB$ induced from $\omega_k$.

\begin{proposition}\label{WP-rho}
Suppose $\omega_X$ restricts to a cscK metric on each fibre. Let $q: B \to \M$ be the moduli map to the moduli space of cscK manifolds. Then $$\int_{X/B} \rho_{\scH}\wedge\omega_X^m = -q^*\Omega_{WP},$$ where $\Omega_{WP} \in c_1(\scL_{CM})$ is the Weil-Petersson metric.
\end{proposition}

\begin{proof} Since $\omega_b$ is cscK for all $b \in B$, a result of Fine provides $$\int_{X/B} \rho_{H} \wedge \omega_X^m = - \int_{X/B}\rho \wedge \omega_X^m + \frac{S(\omega_b)}{m+1}\int_{X/B} \omega^{m+1},$$ where $S(\omega_b)$ is the scalar curvature of any fibre \cite[Lemma 2.3]{fine2}. This is precisely the formula for the Weil-Petersson metric on $\M$, even in the case that of manifolds admitting continuous automorphisms \cite[Theorem 4.4]{DN}.
\end{proof}

Proposition \ref{WP-rho} was essentially proved by Fine \cite[Lemma 2.3]{fine2} when the fibres have discrete automorphism group, and explicitly observed in \cite[Lemma 3.5]{fibrations} in this special case.

\begin{corollary} Suppose all fibres of $X \to B$ are isomorphic. Then $$\int_{X/B} \rho_{\scH}\wedge\omega_X^m =0.$$\end{corollary}

\begin{proof} This follows immediately from the above, since the moduli map is to a point. In fact the statement does not need to apply any of the deeper moduli theory of \cite{DN} to obtain this statement. Instead, one can argue as follows. 

Firstly, note that the integral is independent of choice of fibrewise cscK metric in the class $c_1(H)$ \cite[p. 20]{DN}. Since all fibres are isomorphic, by the Fischer-Grauert Theorem the fibration $X \to B$ is actually locally trivial \cite{fischer-grauert}. Thus if the fibre is $F$, in a local patch $X \cong F \times U$ with $U \subset B$, by independence of the integral of fibrewise cscK metric one can take the relatively K\"ahler metric $p_1^*\omega_F$. Here $p_1$ denotes the projection $F \times U\to U$ and $\omega_F$ is any cscK metric on the fibre $F$. Using such a metric, it is clear by direct calculation that the fibre integral vanishes. \end{proof}

\begin{lemma}\label{ricci-expansion} The Ricci curvature satisfies $$\Ric(\omega_k) = \Ric(\omega_b) + \rho_{H} +  \Ric(\omega_B) +  k^{-1}\ddc  (\Lambda_{\omega_B}\omega_X) + O (k^{-2}).$$ Here $\Ric(\omega_b) = \rho_{\scV}$ restricts to the Ricci curvature of $\omega_b$ on each fibre $X_b$, and we have suppressed pullbacks via $\pi$ as usual.
\end{lemma}

\begin{proof} 
Recall that $\Lambda^n TX  = \Lambda^m V \otimes \pi^* \Lambda^b TB$. Thus the Ricci curvature of $\omega_k$ is the sum of the corresponding curvatures of its vertical and horizontal component. The former is the form $\rho$ defined above. For the latter, we first note that if $iF_{\wedge^b\scH}$ denotes the curvature of the induced metric on $iF^{\wedge^b\scH}$ that

\begin{align*} iF^{\wedge^b\scH} - \Ric(\omega_B) &= \ddc\log\left(\frac{(k\omega_B + (\omega_X)_{\scH})^n}{\omega_B^n}\right), \\ &= \ddc\log(k^n + k^{n-1} \Lambda_{\omega_B} \omega_X+ O(k^{n-2})) .\end{align*} We thus obtain using the power series expansion $\log(1+x) = \sum_{i=1}^{\infty}(-1)^{i+1}(x^i/i)$, valid for $|x|<1$ and hence in our situation for $k \gg 0$, that $$\Ric (\omega_k) = \rho + \Ric(\omega_B) + k^{-1}(\ddc (\Lambda_{\omega_B}\omega_X)) + O(k^{-2}),$$ as desired. \end{proof}

Combining the expansions we have obtained gives the following.

\begin{corollary}\label{cor:scalarexpansion}
$$S(\omega_k) = S(\omega_b) + k^{-1}\big( S(\omega_B) + \Lambda_{\omega_B} \rho_{\scH} + \Delta_{\scV} (\Lambda_{\omega_B}(\omega_X)_{\scH})\big)  + O(k^{-2}).$$ Here $S(\omega_b)$ is the function whose restriction to $X_b$ is the scalar curvature of $\omega_B$, i.e. $S(\omega_b) = \Lambda_{\scV}\Ric(\omega_b).$ 

\end{corollary} 

\begin{proof} This follows immediately from the above, using that $$S(\omega_k) = \Lambda_{\omega_k}\Ric(\omega_k)$$ and the expansions for $\Ric(\omega_k)$ and $\Lambda_{\omega_k}$.  Indeed firstly $$\Lambda_{\omega_B}(\Ric(\omega_k)) =k^{-1}(S(\omega_B) + \Lambda_{\omega_B}\rho_{\scH}) + O(k^{-2}).$$ Next, since the vertical Laplacian of a form $\beta$ is computed by restricting the form to a fibre, and taking the Laplacian on that fibre, it follows that for a function $f$ on $X$, we have $$(\Lambda_{\scV}\ddc f)|_{X_b} = \Lambda_{\omega_b}\ddc (f|_{X_b}),$$ where the partial derivatives of $f$ on the right hand side of the equation are taken on a fibre $X_b$. Thus by definition of the vertical Laplacian$$\Lambda_{\scV}\Ric(\omega_k) = S(\omega_b) + k^{-1}(\Delta_{\scV}(\Lambda_{\omega_B}\omega_X)) + O(k^{-2}),$$ and summing gives the result. \end{proof}

\subsection{The $k^{-1}$ term}\label{subsec:kinverse} We now take $\omega_X$ to be a relatively K\"ahler metric which restricts to a cscK metric on each fibre. Thus the function $S(\omega_b)$ appearing in the expansion of $S(\omega_k)$ is a constant. Our next goal is to show that, with an appropriately chosen choice of $\omega_X, \omega_B$ and $\phi_{R,1} \in C_R^{\infty}(X)$, the $k^{-1}$ coefficient of the scalar curvature $S(\omega_X + k \omega_B + k^{-1}\ddc \phi_{R,1})$ is also constant.

We begin with the correct choice of $\omega_B$. Let $\alpha$ be the Weil-Petersson metric on $B$, so that by Proposition \ref{WP-rho} $$\alpha = -\int_{X/B}\rho_{\scH}\wedge\omega_X^m.$$ We take $\omega_B$ to be a solution of the twisted cscK equation $$S(\omega_B) - \Lambda_{\omega_B} \alpha = const.$$ The fibre metric $\omega_X$ is chosen to be an optimal symplectic connection, so that by definition $$p(\Delta_{\scV}(\Lambda_{\omega_B} \mu^*(F_{\scH})) + \Lambda_{\omega_B} \rho_{\scH})=0.$$

We now explain how this dictates our choice of $\phi_{R,1}$, and in particular how optimal symplectic connections allow us to solve the approximate cscK equation to order $k^{-1}$.

\begin{proposition}\label{prop:firstapprox} Suppose $\omega_B$ is a twisted cscK metric and $\omega_X$ is an optimal symplectic connection. Then $$S(\omega_X + k \omega_B) = \hat S_{0} + k^{-1}(\hat S_1 +  \psi_{R,1}) + O(k^{-2}),$$ with $\psi_{R,1} \in C^{\infty}_R(X)$ and $\hat S_0, \hat S_1 \in \R$.\end{proposition}

\begin{proof} By hypothesis, the function $S(\omega_b)$ is constant and independent of the fibre, hence the statement of the Proposition only concerns the $k^{-1}$ coefficient, which is given by $$S_{1} =  S(\omega_B) + \Lambda_{\omega_B} \rho_{\scH} + \Delta_{\scV}(\Lambda_{\omega_B}(\omega_X)_{\scH}).$$

Note that $$\int_{X/B}(\Lambda_{\omega_B} \rho_{\scH} )\omega_X^m =  \Lambda_{\omega_B}\int_{X/B}\rho_{\scH}\omega_X^m = -\Lambda_{\omega_B} \alpha,$$ with $\alpha$ the Weil-Petersson metric. Thus under the decomposition $$C^{\infty}(X,\R) \cong C^{\infty}(B,\R) \oplus C^{\infty}_E(X,\R) \oplus C_R^{\infty}(X,\R)$$ of equation \eqref{eq:function-decomposition}, the $C^{\infty}(B,\R)$-component of the function $\Lambda_{\omega_B} \rho_{\scH}$ is $\Lambda_{\omega_B} \alpha$, while the $C^{\infty}_E(X,\R)$-component is $p(\Lambda_{\omega_B} \rho_{\scH})$. 

We next consider the term $ \Delta_{\scV}(\Lambda_{\omega_B}(\omega_X)_{\scH})$. Note that $$\int_{X/B} \Delta_{\scV}(\Lambda_{\omega_B}(\omega_X)_{\scH})\wedge\omega_X^m = 0,$$ since $\Delta_{\scV}(\Lambda_{\omega_B}(\omega_X)_{\scH})|_{X_b} = \Delta_{\omega_b}\big( (\Lambda_{\omega_B}\omega_X)|_{X_b} \big)$ integrates to zero over each fibre.

Thus since $\omega_X$ is an optimal symplectic connection and $\omega_B$ is a twisted cscK metric, under the splitting of the space of functions on $X$, both the $C^{\infty}(B,\R)$ and $C^{\infty}_E(X,\R) $ terms are constant, which  is to say $S_1 \in C^{\infty}_R(X,\R)$ up to the addition of a constant. If one adds a term $k^{-1}\ddc\phi_{R,1}$ to $\omega_{k}$, it is clear that this changes the scalar curvature $S(\omega_k)$ by $$S(\omega_X + k \omega_B + k^{-1}\ddc\phi_{R,1}) = S(\omega_X) + k^{-1}(S_1 + \D^*_{\scV}\D_{\scV}\phi_{R,1}) + O(k^{-2}),$$ since $\D^*_{\scV}\D_{\scV}$ is the fibrewise linearisation of the scalar curvature. It follows  from fibrewise ellipticity that the operator $\D^*_{\scV}\D_{\scV}$ is an isomorphism on $C^{\infty}_R(X,\R)$, so there is a choice of $\phi_{R,1}$ such that $\D^*_{\scV}\D_{\scV}\phi_{R,1} = -S_1,$ again up to the addition of a constant. Thus with this choice of $\phi_{R,1}$, the $k^{-1}$-coefficient of $S(\omega_X + k \omega_B + k^{-1} \ddc\phi_{R,1})$ is constant, as required.\end{proof}

\subsection{The linearisation of the optimal symplectic curvature: special case}\label{subsec:linearisation-special} The goal of this section is to understand how the scalar curvature changes upon adding a function $\phi_E \in C^{\infty}_E(X)$, to leading order. Viewing $E$ as a smooth vector bundle over $B$ as in Section \ref{sec:optimal}, we shall see that this essentially asks what the linearisation of the optimal symplectic curvature equation is at a solution. 

We begin by considering the case that $k\omega_B + \omega_X$ actually has constant scalar curvature to $O(k^{-2})$, or equivalently that the $\phi_{R,1}$ term constructed in Section \ref{subsec:kinverse} vanishes. Later in Section \ref{subsec:general} we will return to the general case, using results obtained in the special case. 

Let $\scL_{\omega_k} = \scD_k^* \scD_k$ be the Lichnerowicz operator of $\omega_k = k\omega_B + \omega_X$. This has an expansion $$\L_{\omega_k} = \L_0 + \L_1 k^{-1} + \cdots,$$ where $\L_0 =\D^*_{\scV}\D_{\scV}$ is the glued fibrewise Lichnerowicz operator. We note that $- \L_0 - k^{-1} \L_1$ is the linearisation of the scalar curvature up to order $k^{-1}$, provided our metric is of constant scalar curvature to order $k^{-1}$. Indeed, in general, the linearisation of the scalar curvature is $$ f \mapsto - \scL_{\omega} (f) + \frac{1}{2} \langle \nabla S(\omega), \nabla f \rangle.$$ If $S(\omega_k)$ is constant to order $k^{-1}$, this means the term $ \langle \nabla S(\omega_k), \nabla f \rangle_{\omega_k}$ is $O(k^{-2})$ since $\nabla S(\omega_k)$ is, and so will not appear in the linearisation to order $k^{-1}$.

We now wish to study the mapping properties of $\L_1$ on $C^{\infty}_E (X)$, i.e. on fibrewise average zero functions that are fibrewise holomorphy potentials. We will consider this as a map $E \to E$ by projecting to $C^{\infty}_E (X)$, that is by considering the operator $p \circ \L_1.$ We will exploit the equality 
\begin{align}\label{kintparts} \int_X \phi \scL_{\omega_k} (\psi) \omega_k^n = \int_X \langle \mathcal{D}_k \phi, \mathcal{D}_k \psi \rangle_{\omega_k} \omega_k^n. 
\end{align}

We first note that $$\omega_k^{m+n} = {m+n \choose m } k^n \omega_X^m \wedge \omega_B^n + O(k^{n-1}).$$ So if $\psi$ is in the kernel of $\L_0$, the left hand side of \eqref{kintparts} becomes $${m+n \choose m } k^{n-1} \int_X \phi \L_1 (\psi) \omega_X^m \wedge \omega_B^n  + O(k^{n-2}).$$ We wish to compare this with the $k^{n-1}$-term of the right hand side of \eqref{kintparts}. This is a bit more involved, and we begin with a description of $\mathcal{D}_k = \overline{\partial}_X \circ \nabla_{\omega_k}^{1,0}.$ Note that the $k$-dependency is only in the gradient. 

In general, on an $n$-dimensional K\"ahler manifold with coordinates $\xi^p$ and Riemannian metric $g$, the gradient is $$\nabla_g f = \sum_{p,q} g^{p \overline{q}} \frac{\partial f}{\partial \overline{\xi}^q } \frac{\partial}{\partial \xi^p}.$$
In a holomorphic trivialisation of $X$, we now have $\xi^{\alpha} = w^{\alpha}$, the fibre coordinates, for $\alpha = 1, \cdots, m$, and $\xi^{m+i} = z^i$, the base coordinates, for $i=1, \cdots, n$. We will use Greek letters for the fibre indices and Roman letters for the base indices. The gradient is then $$ \nabla_{\omega_k} f = \sum_{\alpha, \beta}  g_F^{\alpha \overline{\beta}} \frac{\partial f}{\partial \overline{w}^{\beta} } \frac{\partial}{\partial w^{\alpha}} + \sum_{i,j } \big( k g_B + (g_X)_{\scH} \big)^{i \overline{j}} \frac{\partial f}{\partial \overline{z}^j } \frac{\partial}{\partial z^i} .$$ The leading order term is therefore $ \sum_{\alpha, \beta}  g_F^{\alpha \overline{\beta}} \frac{\partial f}{\partial \overline{w}^{\beta} } \frac{\partial}{\partial w^{\alpha}} $ since $$\big( k g_B + (g_X)_{\scH} \big)^{i \overline{j}} = \frac{1}{k} \big( g_B + k^{-1} (g_X)_{\scH} \big)^{i \overline{j}}$$ is $O(k^{-1}).$ If $f$ is a fibrewise holomorphy potential, we then have that $$ \scD_k f = \sum_{\alpha, \beta, j} \frac{\partial}{\partial \overline{z}^j } \left( g_F^{\alpha \overline{\beta}} \frac{\partial f}{\partial \overline{w}^{\beta} } \right)  \frac{\partial}{\partial w^{\alpha}} \otimes d\overline{z}^j+ O(k^{-1}),$$ since the fibrewise components when applying $\overline{\partial}_X$ vanish by assumption. This leading order term will be a crucial operator for us. We define $$\mathcal{R} (f) = \bar{\partial}_B ( \nabla_{\scV}^{1,0} f) = \sum_{\alpha, \beta, j} \frac{\partial}{\partial \overline{z}^j } \left( g_F^{\alpha \overline{\beta}} \frac{\partial f}{\partial \overline{w}^{\beta} } \right)  \frac{\partial}{\partial w^{\alpha}} \otimes d\overline{z}^j.$$ Note that on $C^{\infty}_E$, the kernel of $\mathcal{R}$ consists precisely of the fibrewise holomorphy potentials that in fact are globally holomorphy potentials on $X$.

Next we consider $\langle \cdot , \cdot \rangle_{\omega_k}$, which is a norm on $TX \otimes T^* X.$ We will consider its restriction to $\scV^{1,0} \otimes \pi^* \Lambda^{0,1} B$. From our computation of $\scD_k f$, the key is to compute $\langle \cdot , \cdot \rangle_{\omega_k}$ on basis vectors $\frac{\partial}{\partial w^{\alpha}} \otimes d\overline{z}^j$. First we observe
\begin{align*} \left\langle \frac{\partial}{\partial w^{\alpha}} \otimes d\overline{z}^j , \frac{\partial}{\partial \overline{w}^{\gamma} } \otimes dz^p \right\rangle_{\omega_k} = \left\langle \frac{\partial}{\partial w^{\alpha}} , \frac{\partial}{\partial \overline{w}^{\gamma}} \right\rangle_{\omega_k} \cdot \langle  d\overline{z}^j , d z^p \rangle_{\omega_k} 
\end{align*}
and $$  \left\langle \frac{\partial}{\partial w^{\alpha}} , \frac{\partial}{\partial w^{\gamma}} \right\rangle_{\omega_k} = g^F_{\alpha \overline{\beta}},$$ which is independent of $k$. Moreover, since $$  \left\langle  \frac{\partial}{\partial \overline{z}^j} , \frac{\partial}{\partial z^p} \right\rangle_{\omega_k}  = k g_{\overline{j} p} + O(1),$$ we have  $$ \langle  d\overline{z}^j , d z^p \rangle_{\omega_k}  = k^{-1} g_B^{\overline{j} p} + O(k^{-2}).$$ 

The upshot is therefore that \begin{align*} \left\langle \frac{\partial}{\partial w^{\alpha}} \otimes d\overline{z}^j , \frac{\partial}{\partial \overline{w}^{\gamma} } \otimes dz^p \right\rangle_{\omega_k} = k^{-1}  \left\langle \frac{\partial}{\partial w^{\alpha}} \otimes d\overline{z}^j , \frac{\partial}{\partial \overline{w}^{\gamma} } \otimes dz^p \right\rangle_{\omega_F + \omega_B} + O(k^{-2}),
\end{align*}
and so 
$$\langle \mathcal{D}_k \phi, \mathcal{D}_k \psi \rangle_{\omega_k} = k^{-1} \langle \mathcal{R} \phi , \mathcal{R} \psi \rangle_{\omega_F + \omega_B} + O(k^{-2})$$ for fibrewise holomorphy potentials $\phi, \psi.$ Thus the right hand side of \eqref{kintparts} is $$ {n+m \choose m } k^{n-1} \int_X \langle \mathcal{R} \phi , \mathcal{R} \psi \rangle_{\omega_F + \omega_B}\omega_F^m \wedge \omega_B^n + O(k^{n-2}).$$ Since this holds for all $k$, it follows that 
\begin{align}\label{adjointproperty} \int_X \phi \L_1 (\psi) \omega_X^m \wedge \omega_B^n = \int_X \langle \mathcal{R} \phi , \mathcal{R} \psi \rangle_{\omega_F + \omega_B} \omega_X^m \wedge \omega_B^n .
\end{align}
Note in particular that this implies that $p \circ \L_1$ is self-adjoint.

Next, we wish to see that $p \circ \L_1$ is elliptic. As noted above, the fibrewise Hamiltonian vector fields on $X$, with respect to $\omega_X$, form an infinite-dimensional bundle over $B$. Under our assumptions, there is a finite-dimensional subbundle $E \to B$ whose sections are fibrewise holomorphic vector fields. By identifying a Hamiltonian vector field with its (unique) mean value zero Hamiltonian function, we can think of sections of $E$ over an open set $U \subseteq B$  as certain functions on $\pi^{-1} (U)$, the points in $X$ lying over $U$. In particular, smooth global sections of $E$ correspond exactly to functions in $C^{\infty}_E (X)$. 

That this is a, say rank $r$, vector bundle means that for every point $p \in B$, we can find an open set $U$ containing $p$, and $r$ fibrewise Hamiltonian functions $h_1, \cdots, h_r$ on $\pi^{-1} (U)$ such that any function $h$ in $\pi^{-1} (U)$ which is fibrewise a holomorphy potential (fibrewise of average zero) is of the form $h = \sum_i f_i h_i,$ for some $f_i \in C^{\infty} (U)$. 

We are interested in the mapping properties $p \circ \L_1$, where we recall $p$ is the projection operator to $C^{\infty}_E (X)$. Note, however, that if $\phi, \psi \in C^{\infty}_E(X)$, then $\langle \phi , \L_1 \psi \rangle = \langle \phi, p (\L_1 \psi) \rangle.$ Thus, we will suppress the projection operator in what follows. Equation \eqref{adjointproperty} can now be recast as $$ \int_B \langle \phi , \L_1 \psi \rangle_1 \omega_B^n = \int_B \left( \int_{X/B} \langle \mathcal{R} \phi , \mathcal{R} \psi \rangle_{\omega_F + \omega_B} \omega_X^m  \right) \omega_B^n.$$

We now wish to show that $p \circ \L_1 : \Gamma (E) \to \Gamma (E)$ is elliptic (of order $2$). The key point is that for $f \in C^{\infty} (B)$ and $h \in C^{\infty}_E (X)$, we have \begin{align}\label{symboleqn} \L_1 ( f h ) = 2 \overline{\partial}_B^* \overline{\partial}_B (B) (f) \cdot \Delta_{\scV} (h) + Q(f,p),\end{align} where $Q$ is of order $1$ in $f$. If we establish this, we see that in a trivialisation of $E$ over $U$ given by $r$ fibrewise holomorphy potentials $h_i$, if $\sum_i f_i h_i$ is a general local section of $E$, the symbol of $ \L_1 (\sum_ i f_i h_i ) $ is that of $2 \sum_i \Delta_{B} (f_i) \cdot \Delta_{\scV} (h_i ).$ Moreover, $$p \big( 2 \sum_i \Delta_{B} (f_i) \cdot \Delta_{\scV} (h_i ) \big) = \sum_{i,j} \Delta_{B} (f_i) \cdot \left( \int_{\pi^{-1}(U)/U} \langle \nabla h_i, \nabla h_j \rangle \omega_b^m \right) h_j,$$ since $ 2 \int_{\pi^{-1}(U)/U} \Delta_{\scV} (h_i ) \cdot h_j \omega_b^m = \int_{\pi^{-1}(U)/U} \langle \nabla h_i, \nabla h_j \rangle \omega_b^m.$ The $\nabla h_i$ form a basis of the holomorphic vector fields with a potential on each fibre, and so the matrix $\langle \nabla h_i, \nabla h_j \rangle$ is invertible. The symbol of $p \circ \L_1$ is the composition of this matrix with the symbol of the diagonal component wise base Laplacian $$ (f_1, \cdots, f_r) \mapsto (\Delta_B f_1, \cdots, \Delta_B f_r),$$ which then is elliptic since the Laplacian is.

To establish \eqref{symboleqn}, we will use the general identity $$\mathcal{D}_{\omega}^* \mathcal{D}_{\omega} (\phi ) = \Delta_{\omega}^2 (\phi) + \langle \textnormal{Ric} (\omega) , \ddc \phi \rangle + \langle \nabla S (\omega), \nabla \phi \rangle, $$ see \cite[p. 59]{szekelyhidi-book}. Applying this to $\omega_k$, we can ignore the last term involving the scalar curvature because it only involves a single derivative. The $\Delta^2$-term expands as $$\Delta_k^2 = \Delta_{\scV}^2 + k^{-1} ( \Delta_{\scV} \circ \Delta_{\scH} + \Delta_{\scH} \circ \Delta_{\scV}) + O (k^{-2}),$$ since $\Delta_k = \Delta_{\scV} + k^{-1} \Delta_{\scH} + O(k^{-2}).$ For a function of the form $f \cdot h$ with $f$ pulled back from $B$, we have  $$\Delta_{\scV}(fh) = f \Delta_{\scV} (h).$$ Moreover, $\Delta_{\scH} (fh)$ equals $\Delta_{\scH} (f) \cdot h$ plus terms which involve at most one derivatives in $f$. So the contribution to the $k^{-1}$ term from $\Delta^2_k$ with the maximal number of derivatives in $f$ is then as claimed.

For the term involving the Ricci curvature of $\omega_k$, we see from the expansion of the Ricci curvature in Lemma \ref{ricci-expansion} that we will get a contribution $$ \langle \rho_{\scH} +  \Ric(\omega_B) , \ddb_{\scH} f \rangle_{\omega_k} \cdot h$$ that involves two derivatives of $f$. But since the inner product in the horizontal direction is $O(k)$, the induced inner product on $\Lambda^2$ is $O(k^{-2}).$ Thus we will not see this term in the linearisation to order $k^{-1}$. This completes the proof of the identity \eqref{symboleqn}, and hence the proof of the ellipticity of $p \circ L_1$ thought of as an operator on $E$. 

We also remark that the technique above implies that $$ \int_{X} \pi^*(\phi) \L_1 (\psi) \omega_X^m \wedge \omega_B^n = \int_B \phi\left( \int_{X/B} \L_1(\psi)\omega_b^m\right)\omega_B ^n= 0 $$ for any $\phi \in C^{\infty} (B)$, since $\int_X \pi^* (\phi) \mathcal{L}_{\omega_k}(\psi)  \omega_k^n = \int_X \mathcal{L}_{\omega_k}(\pi^* (\phi)) \psi  \omega_k^n$ is $O(k^{n-2})$ (the fact that this is $O(k^{n-2})$ rather than $O(k^{n-1})$ will follow from Proposition \ref{prop:linearisation}). So $\L_1 (\psi)$ is always orthogonal to $C^{\infty} (B)$. This shows that $\L_1$, without the projection $p$, is a vertical operator: it send $C^{\infty}(X)$ to $C^{\infty}_0 (X)$. 

We have now established:
\begin{theorem}\label{thm:mappingprops} Consider the operator $p \circ \L_1 : \Gamma (E) \to \Gamma(E)$. This is an elliptic operator of order two which is self-adjoint and with kernel consisting precisely of sections corresponding to global holomorphy potentials on $X$. 
\end{theorem}

\subsection{The linearisation of the optimal symplectic curvature: general case}\label{subsec:general} In Section \ref{subsec:linearisation-special} we established the properties of the linearisation of the optimal symplectic curvature in the special case that $k\omega_B+\omega_X$ is a constant scalar curvature K\"ahler metric to order $k^{-2}$. This is equivalent to asking that the function $\phi_{R,1}$ constructed in Section \ref{subsec:kinverse} vanishes.

As remarked earlier, $- \mathcal{L}_{\omega_k}$ is the linearised operator to order $k^{-1}$ only if the metric is constant to order $k^{-1}$. This is not automatic even if we have an optimal symplectic connection, and in general we need to consider instead a metric of the form $\Omega_k = \omega_k + k^{-1} \ddc \phi$, where in our applications $\phi$ will be the function  $\phi_{R,1}$ constructed in Section \ref{subsec:kinverse}. We will show that the Lichnerowicz operator is the same on $E$, to highest order in $k$ and so the mapping properties described above persist under such a change to $\omega_k$. As before, we will use
\begin{align}\label{kintparts2} \int_X \phi \scL_{\Omega_k} (\psi) \Omega_k^n = \int_X \langle \mathcal{D}_k \phi, \mathcal{D}_k \psi \rangle_{\Omega_k} \Omega_k^n. 
\end{align}
Here $\mathcal{D}_k$ is now the operator associated to $\Omega_k$, not $\omega_k$.

Again we have that $$\Omega_k^n = {n+m \choose m } k^n \omega_X^m \wedge \omega_B^n + O(k^{n-1}).$$ So if $\psi$ is in the kernel of $\L_0$, the left hand side of \eqref{kintparts} becomes $${n+m \choose m } k^{n-1} \int_X \phi \L_1 (\psi) \omega_X^m \wedge \omega_B^n  + O(k^{n-2}).$$ So there is no change in the left hand side of the equation. 

The next step is to compare this with the $k^{n-1}$-term of the right hand side of \eqref{kintparts2}. We begin with a description of $\mathcal{D}_k = \overline{\partial}_X \circ \nabla_{\Omega_k}^{1,0}.$ Note that the dependency on $k$ and $\phi$ is only in the gradient. The gradient is then $$ \nabla_{\Omega_k} f = \sum_{\alpha, \beta}  (g_F + k^{-1} \ddc \phi)^{\alpha \overline{\beta}} \frac{\partial f}{\partial \overline{w}^{\beta} } \frac{\partial}{\partial w^{\alpha}} + \sum_{i,j } \big( k g_B + (g_X)_{\scH} \big)^{i \overline{j}} \frac{\partial f}{\partial \overline{z}^j } \frac{\partial}{\partial z^i} .$$ The leading order term is therefore $ \sum_{\alpha, \beta}  g_F^{\alpha \overline{\beta}} \frac{\partial f}{\partial \overline{w}^{\beta} } \frac{\partial}{\partial w^{\alpha}} $ since the inverse of $$g_F + k^{-1} \ddc \phi $$ is $g_F^{-1} + O(k^{-1})$ and $$\big( k g_B + (g_X)_{\scH} \big)^{i \overline{j}} = \frac{1}{k} \big( g_B + k^{-1} (g_X)_{\scH} \big)^{i \overline{j}}$$ is $O(k^{-1}).$ If $f$ is a fibrewise holomorphy potential, we then have that $$ \scD_k f = \sum_{\alpha, \beta, j} \frac{\partial}{\partial \overline{z}^j } \big( g_F^{\alpha \overline{\beta}} \frac{\partial f}{\partial \overline{w}^{\beta} } \big)  \frac{\partial}{\partial w^{\alpha}} \otimes d\overline{z}^j+ O(k^{-1}),$$ since the fibrewise components when applying $\overline{\partial}_X$ vanish by assumption. Note that there is no dependency on $\phi$ here, so we have the same operator as before. 

The argument for $\langle \cdot , \cdot \rangle_{\Omega_k}$ proceeds in a similar way. First we again observe
\begin{align*} \left\langle \frac{\partial}{\partial w^{\alpha}} \otimes d\overline{z}^j , \frac{\partial}{\partial \overline{w}^{\gamma} } \otimes dz^p \right\rangle_{\Omega_k} = \left\langle \frac{\partial}{\partial w^{\alpha}} , \frac{\partial}{\partial \overline{w}^{\gamma}} \right\rangle_{\Omega_k} \cdot \langle  d\overline{z}^j , d z^p \rangle_{\Omega_k} 
\end{align*}
and $$  \left\langle \frac{\partial}{\partial w^{\alpha}} , \frac{\partial}{\partial w^{\gamma}} \right\rangle_{\Omega_k} = \left(g^F+\frac{1}{k} \ddc \phi\right)_{\alpha \overline{\beta}} = g^F_{\alpha \overline{\beta}} + O(\frac{1}{k}),$$ which is independent of $k$ to leading order. Moreover, we still have $$  \left\langle  \frac{\partial}{\partial \overline{z}^j} , \frac{\partial}{\partial z^p} \right\rangle_{\Omega_k}  = k g_{\overline{j} p} + O(1),$$ as our horizontal metric is now $k g_B + (\omega_X)_{\scH} + \frac{1}{k} \ddc \phi_{\scH}$. Thus we have  $$ \langle  d\overline{z}^j , d z^p \rangle_{\omega_k}  = k^{-1} g_B^{\overline{j} p} + O(k^{-2}).$$  

The upshot is therefore that \begin{align*} \left\langle \frac{\partial}{\partial w^{\alpha}} \otimes d\overline{z}^j , \frac{\partial}{\partial \overline{w}^{\gamma} } \otimes dz^p \right\rangle_{\omega_k} = k^{-1}  \left\langle \frac{\partial}{\partial w^{\alpha}} \otimes d\overline{z}^j , \frac{\partial}{\partial \overline{w}^{\gamma} } \otimes dz^p \right\rangle_{\omega_F + \omega_B} + O(k^{-2}),
\end{align*}
and so 
$$\left\langle \mathcal{D}_k \phi, \mathcal{D}_k \psi \right\rangle_{\omega_k} = k^{-1} \langle \mathcal{R} \phi , \mathcal{R} \psi \rangle_{\omega_F + \omega_B} + O(k^{-2})$$ for fibrewise holomorphy potentials $\phi, \psi.$ Thus the right hand side of \eqref{kintparts} is $$ {n \choose m } k^{n-1} \int_X \langle \mathcal{R} \phi , \mathcal{R} \psi \rangle_{\omega_F + \omega_B}\omega_F^m \wedge \omega_B^n + O(k^{n-2}).$$ Since this holds for all $k$, it follows that
\begin{align}\label{adjointproperty} \int_X \phi L_1 (\psi) \omega_X^m \wedge \omega_B^n = \int_X \langle \mathcal{R} \phi , \mathcal{R} \psi \rangle_{\omega_F + \omega_B} \omega_X^m \wedge \omega_B^n .
\end{align}
Note in particular that this implies that $p \circ \L_{1,\phi}$ is self-adjoint and in fact independent of $\phi$. So the operator remains unchanged upon this perturbation.

To summarise, we have just shown that the second term in the expansion of the Lichnerowicz operator is unaffected by the addition of $k^{-1} \ddc \phi$, when acting on fibrewise holomorphy potentials (with respect to $\omega_F$, but to leading order these are the same as holomorphy potentials with respect to $\omega_F + k^{-1} \ddc _{\scV} \phi$, where the subscript means the fibrewise restriction of $\ddc \phi$). When $\phi$ is chosen such that $\Omega_k$ is a cscK metric to order $k^{-1}$ (meaning $S(\Omega_k) = S_0 + S_1 k^{-1} + O(k^{-2})$), we saw that the linearisation of the scalar curvature map is the negative of the Lichnerowicz operator to order $k^{-1}$, and so in particular for that choice of $\phi$, the above holds for the linearised operator, not just the Lichnerowicz operator. The argument above means that we have proven the case $r=1$ in Proposition \ref{prop:linearisation} below, which will be crucial for constructing better approximate solutions to the cscK equation.

\subsection{The discrete automorphism group case} While we are ultimately interested in the construction of extremal metrics on $(X,kL+H)$ for $k \gg0$, we begin by discussing the case that the automorphism groups $\Aut(X,H)$ and $\Aut(p)$ of $(X,H)$ and $p: B \to \M$ are discrete. This special case contains the bulk of the core new ideas of our construction, and in this case the argument is technically and notationally simpler. Our next goal is to construct, for any given $r \geq 0$, a K\"ahler potential $\phi_{k,r} \in C^{\infty}(X)$ which satisfies to the equation $$S(\omega_k + \ddb \phi_{k,r}) - \hat S_{k,r}= O(k^{-r}),$$ with $\hat S_{k,r}$ an appropriate topological constant.

\begin{theorem}\label{thm:cscKapprox} Let $\pi : (X,H) \to (B,L)$ be a fibration with $\Aut(X,H)$ discrete. Suppose
\begin{enumerate}[(i)]
\item $\omega_X \in c_1(H)$ restricts to a cscK metric on $X_b$ for all $b\in B$;
\item letting $q: B \to \M$ be the moduli map, $\Aut(q)$ is discrete and $\omega_B$ satisfies the twisted cscK equation  $$S(\omega_B) - \Lambda_{\omega_B}\alpha = \hat S_{\alpha}$$  with $\alpha$ the the Weil-Petersson metric on $B$;
\item the automorphism groups of the fibres $(X_b,H_b)$ all have the same dimension, and $\omega_X$ satisfies the optimal symplectic connection equation $$p(\Delta_{\scV}(\Lambda_{\omega_B} \mu^*(F_{\scH})) + \Lambda_{\omega_B} \rho_{\scH})=0.$$
\end{enumerate}Then for each integer $r \geq 1$ there exists functions $$f_1, \cdots , f_r \in C^{\infty}(B), \qquad d_1,\hdots, d_{r} \in C^{\infty}_E(X),\qquad l_1, \cdots , l_r \in C^{\infty}_R (X)$$ and constants $c_0, \cdots, c_r$ such that if we let $$\phi_{k,r} = \sum_{j=1}^{r} f_j k^{2-j}, \quad \delta_{k,r} =  \sum_{j=1}^r d_j k^{1-j}, \quad \lambda_{k,r} = \sum_{j=1}^r l_j k^{-j}, \quad \hat S_{k,r} = \sum_{j=0}^r c_j k^{-j},$$ then the K\"ahler form $$\omega_{k,r} = k\omega_B + \omega_X + \ddb ( \phi_{k,r} +\delta_{k,r} +  \lambda_{k,r} )$$ satisfies $$S( \omega_{k,r})  = \hat S_{k,r}  + O(k^{-r-1}).$$
\end{theorem} 

The basic idea of the proof is to use an inductive argument. Supposing the theorem has been proven for level $r$, one writes $$ S( \omega_{k,r}) = \hat S_{k,r} + (\psi_{B,r+1} + \psi_{E,r+1} + \psi_{R,r+1})k^{r+1} + O(k^{r+2}) $$ under the decomposition of $$\psi_{r+1} = \psi_{B,r+1} + \psi_{E,r+1} + \psi_{R,r+1} \in C^{\infty}(X) \cong C^{\infty}(B) \oplus C^{\infty}_E(X) \oplus C_R^{\infty}(X) $$ of equation \eqref{eq:function-decomposition} and using the obvious notation. The next step is to add K\"ahler potentials $f_{r+1}, d_{r+1}$ and $l_{r+1}$ in such a way that the scalar curvature of $$S( \omega_{k,r}+\ddb (f_{r+1}, d_{r+1} + l_{r+1})) =\hat S_{k,r+1}  + O(k^{-r-2}).$$ As we are only interested in the first order change of the scalar curvature upon adding such K\"ahler potentials, the crucial point is to understand the linearisation of $S(\omega_{k,r})$, which to first order is intimately related to the linearisation of $S(\omega_{k})$. In turn, this linearisation is intimately related to the linearisation of the twisted scalar curvature, the optimal symplectic curvature operator and also the fibrewise scalar curvature: the fact that we are linearisating at solutions to the various equations will play a crucial role in ensuring the operators are invertible on the relevant functions spaces. We note that the base case $r=1$ holds by choosing $f_1 =0$ and $d_1 =0$. 

In order to kill the error terms, we need to understand the linearisation of the scalar curvature at $\omega_{k,r}$ and not just $\omega_k$ itself.
\begin{proposition}\label{prop:linearisation} Let $ \omega_{k, r} $ be a metric of the form in the statement of Theorem \ref{thm:cscKapprox} that solves the cscK equation to order $-1$, with $f_1=0=d_1$. Then the linearisation $\mathcal{L}_{k,r}$ of the scalar curvature operator at $\omega_{k,r}$ satisfies the following.
\begin{enumerate}[(i)] \item There is an expansion $$ \mathcal{L}_{k,r} (f) =  - \D^*_{\scV}\D_{\scV}(f) + D_1 (f) k^{-1} + D_2 (f) k^{-2} + O(k^{-3}),$$ where $D_1$ depends on $l_1$ and $D_2$ depends only on $l_1, l_2$ and $d_2$.
\item If $f \in C^{\infty}(B)$, then $D_1 (f)=0$ and $$\int_{X / B} D_2 (f) \omega_X^{m} = - \mathcal{L}_{\alpha} (f) .$$
\item If $f \in C^{\infty}_E(X)$, then $$p \circ D_1 (f) = - p \circ \mathcal{L}_1 (f).$$
\end{enumerate}
\end{proposition}
\begin{proof} Parts (i) and (ii) are not new, and were considered in \cite[Proposition 5.6]{fibrations} building on \cite[Section 3.3]{fine1}. For (iii), we note that we already know the result for the case $r=1$. This was the proof after Theorem \ref{thm:mappingprops} above. We will now make a further perturbation to $\omega_k$. To highest order in the vertical direction, we still get the fibre metric $\omega_F$ (because the only order zero change is through the pullback of a function from $B$). Also, to highest order in the horizontal direction we still get $k \omega_B$, because in the horizontal direction the highest order change is $k \omega_B$ and no other terms in the expansion involves a positive power of $k$. Going through the proof of the expression of $p \circ \L_1$ on functions in $C^{\infty}_E(X)$ again, we see that the expressions only depend on these highest order components. Therefore, they will not change under the types of perturbations we are considering. 
\end{proof}

We will now construct approximate solutions $\omega_{k,r}$ to the cscK equation on $X$. We explain the first step before proceeding to the general case. Let $$\omega_{k,1} =k\omega_B + \omega_X + k^{-1} \ddc l_1$$ be the solution of the cscK equation to order $-1$, so that $$S(\omega_{k,1}) = c_0 + c_1 k^{-1} + k^{-2}(\psi_{B,2} + \psi_{E,2} + \psi_{R,2}) + O(k^{-3})$$ for $\psi_{B,2} \in C^{\infty} (B), \psi_{E,2} \in C^{\infty}_E (X)$ and $\psi_{R,2} \in C^{\infty}_R (X)$.  From Proposition \ref{prop:linearisation}, we know that $$ S(\omega_{k,1} + \ddc f_2 ) = S(\omega_{k,1}) + k^{-2} D_2 (f) + O(k^{-3}),$$ and so $$ S(\omega_{k,1} + \ddc f_2 ) = c_0 + c_1 k^{-1}  + k^{-2}( (\psi_{B,2} - \mathcal{L}_{\alpha} (f_2) ) + \psi'_{E,2} + \psi'_{R,2}) + O(k^{-3})$$ for some $\psi'_{E,2} \in C^{\infty}_E (X)$ and $\psi'_{R,2} \in C^{\infty}_R (X)$ depending on $f_2$, using that the horizontal component of $D_2(f_2)$ is $- \mathcal{L}_{\alpha} (f_2)$. Under our assumptions regarding the automorphisms of the map to the moduli space, $\mathcal{L}_{\alpha}$ is invertible modulo constants, and so there is a constant $c_{2}$ and choice of $f_2$ such that $\psi_{B,2} - \mathcal{L}_{\alpha} (f_2)  = c_{2}$. 

Next, we remove the error in the $E$-component. Proposition \ref{prop:linearisation} gives $$ S(\omega_{k,1} + \ddc f_2 + k^{-1} \ddc d_2 ) = c_0 + k^{-1} c_1 + k^{-2}( c_2 + (\psi'_{E,2} - p \circ \L_1 (\delta_2) ) + \psi''_{R,2}) + O(k^{-3}),$$ where $\psi''_{R,2} \in C^{\infty}_R (X)$ depends on $d_2$. Note since $\L_1$ is a vertical operator, we are not introducing a new error in the base direction to order $k^{-2}.$ Since there are no global automorphisms on $X$, we know from Theorem \ref{thm:mappingprops} that $p \circ \L_1$ is invertible on $E$, and so we can pick $d_2$ such that $\psi'_{E,2} = p \circ \L_1 (d_2)$.

It only remains to kill the $R$-component of the error. This can be achieved similarly as the vertical Lichnerowicz operator $ - \D^*_{\scV}\D_{\scV}$ restricts to an isomorphism on $R$. So we can find $l_2$ such that $$ S(\omega_{k,1} + \ddc f_2 + k^{-1} \ddc d_2  + k^{-2} \ddc l_2  ) = c_0 + c_1 k^{-1} + c_2 k^{-2} + O(k^{-3}),$$ which proves the $r=2$ case of the Theorem.

We are now ready to prove Theorem \ref{thm:cscKapprox}. 
\begin{proof} The proof is by induction, the base step being $r=1$ where we have already noted that the Theorem holds. Suppose $r \geq 1$ and that we have functions $$f_1, \cdots , f_r \in C^{\infty}(B), \qquad d_1,\hdots, d_{r} \in C^{\infty}_E(X),\qquad l_1, \cdots , l_r \in C^{\infty}_R (X)$$ and constants $c_0, \cdots, c_r$ such that if $$\phi_{k,r} = \sum_{j=1}^{r} f_j k^{2-j}, \quad \delta_{k,r} =  \sum_{j=1}^r d_j k^{1-j}, \quad \lambda_{k,r} = \sum_{j=1}^r l_j k^{-j}, \quad \hat S_{k,r} = \sum_{j=0}^r c_j k^{-j},$$ then $$\omega_{k,r} = k\omega_B + \omega_X + \ddb ( \phi_{k,r} +\delta_{k,r} +  \lambda_{k,r} )$$ satisfies $$S( \omega_{k,r})  = \hat S_{k,r}  + O(k^{-r-1}).$$ We begin by dividing the $O(k^{-r-1})$ error into its three components
$$S(\omega_{k,r}) = \hat S_{k,r} + k^{-r-1}(\psi_{B,r+1} + \psi_{E,r+1} + \psi_{R,r+1}) + O(k^{-r-2})$$ for $\psi_{B,r+1} \in C^{\infty} (B), \psi_{E,r+1} \in C^{\infty}_E (X)$ and $\psi_{R,r+1} \in C^{\infty}_R (X)$.  

We begin with removing the horizontal error. Using the mapping properties of the linearised operator on pulled back functions, we have from Proposition \ref{prop:linearisation} that $$ S(\omega_{k,r} + k^{-r+1} \ddc f_{r+1} ) = \hat S_{k,r} + k^{-r-1}\big( \psi_{B,r+1} - \mathcal{L}_{\alpha} (f_{r+1}) )+ \psi'_{E,r+1} + \psi'_{R,r+1}) + O(k^{-r-2})$$ for some $\psi'_{E,r+1} \in C^{\infty}_E (X)$ and $\psi'_{R,r+1} \in C^{\infty}_R (X)$ depending on $f_{r+1}$. Since $\mathcal{L}_{\alpha}$ is invertible modulo constants, there is a constant $c_{r+1}$ and choice of $f_{r+1}$ such that $\psi_{B,2} - \mathcal{L}_{\alpha} (f_2)  = c_{r}$. Thus $$ S(\omega_{k,r} + k^{-r+1} \ddc f_{r+1} ) = \hat S_{k,r+1} + k^{-r-1}\big( \psi'_{E,r+1} + \psi'_{R,r+1}) + O(k^{-r-2}),$$ where $\hat S_{k,r+1} = \hat S_{k,r} + c_{r+1} k^{-r-1}.$

For the $E$-component, Proposition \ref{prop:linearisation} gives us that \begin{align*}S\big(\omega_{k,r} +&  \ddc \big( k^{-r+1} f_{r+1}+ k^{-r} d_{r+1} \big) \big) \\ &= \hat S_{k,r+1} + k^{-r-1}\big( (\psi'_{E,r+1} - p \circ \L_1 (d_{r+1} ) ) + \psi''_{R,r+1} \big) + O(k^{-r-2}),\end{align*} for some $\psi''_{R,r+1} \in C^{\infty}_R (X)$ depending on $d_{r+1}$ and $f_{r+1}$. From Theorem \ref{thm:mappingprops} the operator $p \circ \L_1$ is invertible on $E$, and so we can pick $d_{r+1}$ such that $\psi'_{E,r+1} = p \circ \L_1 (d_{r+1})$. With this choice of $d_{r+1}$, we have $$ S\left(\omega_{k,r} + \ddc \left( k^{-r+1} f_{r+1}+ k^{-r} d_{r+1} \right) \right) = \hat S_{k,r+1} + k^{-r-1}  \psi''_{R,r+1} + O(k^{-r-2}).$$

Finally, since the vertical Lichnerowicz operator is invertible on $R$,  there is an $l_{r+1}$ such that $ - \D^*_{\scV}\D_{\scV} (l_{r+1}) = \psi''_{R,r+1}$, and so $$ S\bigg(\omega_{k,r} + \ddc \big( k^{-r+1} f_{r+1} + k^{-r} d_{r+1} +k^{-r} l_{r+1} \big) \bigg) = \hat S_{k,r+1} + O(k^{-r-2}),$$ as required.
\end{proof}

\subsection{The case of automorphisms}\label{sec:general-case}

We now allow both $q: (B,H) \to \M$ and $(X,H)$ to have automorphisms. In this case, in general we will produce extremal metrics rather than cscK metrics. The presence of automorphisms introduces two main new issues we now have to deal with. The first is that we have a twisted extremal metric on the base, rather than a twisted cscK metric. This is completely analogous to the situation in \cite{fibrations}. The second is that $(X,H)$ may now have global holomorphic vector fields that restrict to non-trivial holomorphic vector fields on each fibre. This will cause a change in the approximation procedure, when dealing with the vertical terms coming from the cokernel of the Lichnerowicz operator on each fibre. 

Since we are aiming to produce extremal metrics, we only need that the symplectic curvature term is the potential for a global holomorphic vector field on $X$, not that it vanishes. Recall that in Section \ref{sec:extremal-symplectic} we defined $\omega_X$ to be an extremal symplectic connection if the function $$p(\Delta_{\scV}(\Lambda_{\omega_B} \mu^*(F_{\scH})) + \Lambda_{\omega_B} \rho_{\scH})$$ is a global holomorphy potential on $X$ with respect to $\omega_k$ for all $k$. By the results we established in Section \ref{subsec:linearisation-special}, we immediately obtain the following:

\begin{lemma} $\omega_X$ is an extremal symplectic connection if and only if $$\mathcal{R} \circ p(\Delta_{\scV}(\Lambda_{\omega_B} \mu^*(F_{\scH})) + \Lambda_{\omega_B} \rho_{\scH})=0.$$ \end{lemma} 

Note below that even if we start with a solution to the optimal symplectic connection equation, and not just its extremal version, through our construction we may still end up with non-cscK extremal metrics on $X$. 

We briefly recall from Section \ref{sec:preliminaries} how to adapt the problem of finding a cscK metric (at the level of finding the appropriate K\"ahler potential) to finding an extremal metric. Recall that $\overline{\mathfrak{h}}$ denotes the space of holomorphy potentials with respect to a fixed K\"ahler metric $\omega$. Solving the extremal equation is then equivalent to finding a root of the map $ C^{k+4, \alpha} \times \overline{\mathfrak{h}} \rightarrow C^{k,\alpha}$ given by 
\begin{align*} (\phi, f ) \to S(\omega_{\phi}) - f -   \frac{1}{2} \langle \nabla f, \nabla \phi \rangle.
\end{align*} 
The  linearisation of this map at $(0,f)$ is $$ (\phi,  h) \to - \mathcal{D}^* \mathcal{D} (\phi) +  \frac{1}{2} \langle \nabla (S(\omega) - f), \nabla \phi \rangle -h.$$
In particular, if $\omega$ is extremal, we can linearise at $(0, S(\omega) )$ and the linearised map is then $$ (\phi, h ) \to - \mathcal{D}^* \mathcal{D} (\phi)  - h,$$ which is a surjective operator with kernel $\overline{\mathfrak{h}} \times \{ 0 \}.$ This is the underlying point of view which will allow us to deal with the seemingly troublesome looking extra terms appearing in the linearisation when producing the approximately extremal metrics below.

\begin{remark}\label{newtorus} In order to establish the extremal version of our main result, namely when the base admits a twisted extremal metric or $\omega_X$ is an extremal symplectic connection, it will be important to assume that $\omega_X$ satisfies two invariance properties. These are not necessary under the assumption that $\omega_X$ is an optimal symplectic connection and $\omega_B$ is a twisted cscK metric. We will use these assumptions to ensure the relevant linearised operators are real operators; in general they may be complex operators.

The first necessary assumption is that $\omega_X$ is invariant under the flow of $$\nu_E = J\nabla_{\scV}(p(\Delta_{\scV}(\Lambda_{\omega_B} \mu^*(F_{\scH})) + \Lambda_{\omega_B} \rho_{\scH})),$$ the vector field arising from the fact that $\omega_X$ is an extremal symplectic connection. We expect that this is always satisfied for extremal symplectic connections, much as extremal K\"ahler metrics are invariant under a maximal compact torus of automorphisms by Theorem \ref{thm:torus-invariance}. The second is that we must assume $\omega_X$ is invariant under the flow of a lift of the twisted extremal vector field $$\nu_q=J\nabla_B (S(\omega_B) - \Lambda_{\omega_B}q^*\Omega)$$ on $B$. In the setting of projective bundles, this invariance automatically follows by uniqueness results for (direct sums of) Hermite-Einstein metrics \cite[Lemma 4.9]{lu-seyyedali}, and hence we expect that this assumption holds automatically for any extremal symplectic connection. 

Thus what is most convenient for us is to fix a torus $T_q$ containing the flow of the twisted extremal vector field  on $B$, and a lift of $T_q$ to a subtorus of $\Aut(X,H)$. We fix also a torus $T_E$ containing the flow of $\nu_E$, and fix a torus $T \subset \Aut(X,H)$ containing both $T_E$ and the lift of $T_q$. It follows from these assumptions and Theorem \ref{isometriesoftwistedextremal} that $\omega_k = \omega_X + k\omega_B$ is a $T$-invariant K\"ahler metric for all $k \gg 0$. We remark again that this assumption is vacuous in the case when $\omega_B$ is twisted cscK and $\omega_X$ is an optimal symplectic connection.
 \end{remark}

It will be important to keep track of the various different holomorphy potentials occuring below. We will refer to the holomorphy potentials with respect to the base metric $\omega_B$ that lift to $X$  as \textit{base holomorphy potentials}, and the functions corresponding to sections of $E$ which produce holomorphy potentials on the whole of $X$ as \textit{global fibre holomorphy potentials}. All of our holomorphy potentials will be chosen to lie in the Lie algebra of the  complexification of $T$, which admits a splitting since we have chosen a lift of $T_q$ to a subtorus of $\Aut(X,H)$.  Note that every holomorphy potential on $X$ with respect to $\omega_k$ is then uniquely decomposed into a sum of a global fibre holomorphy potential and the lift of a base holomorphy potential, see Proposition \ref{automsofX}. For the holomorphic vector fields coming from the base, the holomorphy potentials with respect to the initial metric $\omega_k$ are related to, but not equal to the pullback of the corresponding base holomorphy potential. In fact, they are of the form $\tilde b_k = k \pi^* b + h_b$, for a fixed function $h_b$. Below we will let $\tilde b$ denote the lifted holomorphy potential of $b$, omitting the dependence on $k$. Note that terms $k^{-j-1} \tilde b$ are then $O(k^{-j})$, not $O(k^{-j-1})$.

We now state the main result of this section. 
\begin{theorem}\label{thm:extremalapprox} Let $\pi : (X,H) \to (B,L)$ be a smooth fibration. Suppose
\begin{enumerate}[(i)]
\item $\omega_X \in c_1(H)$ restricts to a cscK metric on $X_b$ for all $b\in B$;
\item letting $q: B \to \M$ be the moduli map, all elements of $\Aut(q)$ lift to $X$ and $\omega_B$ satisfies the twisted extremal equation  $$\mathcal{D}_{\omega_B} \big( S(\omega_B) - \Lambda_{\omega_B}\alpha \big) = 0$$  with $\alpha$ the the Weil-Petersson metric on $B$;
\item the automorphism groups of the fibres $(X_b,H_b)$ all have the same dimension, $\omega_X$ is $T$-invariant and satisfies the extremal symplectic connection equation $$\mathcal{R} \circ p(\Delta_{\scV}(\Lambda_{\omega_B} \mu^*(F_{\scH})) + \Lambda_{\omega_B} \rho_{\scH})=0.$$
\end{enumerate}Then for each integer $r \geq 1$ there exists functions $$f_1, \cdots , f_r \in C^{\infty}(B)^{T}, \qquad d_1,\hdots, d_{r} \in C^{\infty}_E(X)^T, \qquad l_1, \cdots , l_r \in C^{\infty}_R (X)^T,$$ base holomorphy potentials  $$b_1, \cdots, b_r \in C^{\infty} (B)^{T},$$  fibre holomorphy potentials $$ h_1, \cdots, h_r \in C^{\infty}_E(X)^T$$ and a constant $c$ such that if we let $$\phi_{k,r} = \sum_{j=1}^{r} f_j k^{2-j}, \quad \delta_{k,r} =  \sum_{j=1}^r d_j k^{1-j}, \quad \lambda_{k,r} = \sum_{j=1}^r l_j k^{-j},$$ and $$ \eta_{k,r} = c + \sum_{j=1}^r (\tilde b_j k^{-j-1}+ h_j k^{-j}),$$ then the K\"ahler form $$\omega_{k,r} = k\omega_B + \omega_X + \ddb ( \phi_{k,r} +\delta_{k,r} +  \lambda_{k,r} )$$ satisfies $$S( \omega_{k,r})  = \eta_{k,r} + \frac{1}{2} \langle \nabla \eta_{k,r}, \nabla (\phi_{k,r} +\delta_{k,r} +  \lambda_{k,r}) \rangle_{\omega_k} + O(k^{-r-1}).$$
\end{theorem} 

Here $T$ is the compact torus described in Remark \ref{newtorus}; in particular if $\omega_X$ is an optimal symplectic connection and $\omega_B$ is twisted cscK, no invariance assumption is needed and $T$ can be taken to be trivial.

The strategy of the proof is the same as for Theorem \ref{thm:cscKapprox}, with some additional complications due to the presence of the automorphisms. We begin by explaining the step $r=1$, then explain how this affects the linearisation of the scalar curvature operator. Once this is in place we will be able to follow the same steps as before to complete the proof, with some additional care to deal with the extra terms coming from the change in holomorphy potentials.

Recall from Corollary \ref{cor:scalarexpansion} that the scalar curvature of $\omega_k$ has an expansion
$$S(\omega_k) = S(\omega_b) + k^{-1}\big( S(\omega_B) + \Lambda_{\omega_B} \rho_{\scH} + \Delta_{\scV} (\Lambda_{\omega_B}(\omega_X)_{\scH})\big)  + O(k^{-2}).$$  We know $S(\omega_b)$ is a constant, which will be denoted $c$, and that $ S(\omega_B) - \Lambda_{\omega_B} \alpha $ is a holomorphy potential on $B$ with respect to $\omega_B$, which will be the term $b_1$ (recall from the proof of Proposition \ref{prop:firstapprox} that $- \Lambda_{\omega_B} \alpha$ is the base component of $\Lambda_{\omega_B} \rho_{\scH}$ in the decomposition of functions on $X$). Finally, since $$\mathcal{R} \circ p(\Delta_{\scV}(\Lambda_{\omega_B} \mu^*(F_{\scH})) + \Lambda_{\omega_B} \rho_{\scH}) =0$$ we have that  \begin{equation}\label{h1}\Delta_{\scV}(\Lambda_{\omega_B} \mu^*(F_{\scH})) + \Lambda_{\omega_B} \rho_{\scH}) = h_1 + \psi_{R,1}\end{equation} for a global holomorphy potential $h_1 \in C^{\infty}_E (X)$ and some function $ \psi_{R,1} \in C^{\infty}_R (X)$. Thus we have $$S(\omega_k) = c + k^{-1}\big( b_1 + h_1 + \psi_{R,1} \big)  + O(k^{-2}).$$ Since the linearisation of the scalar curvature of $\omega_k$ is $ - \mathcal{D}^*_{\scV} \mathcal{D}_{\scV} + O(k^{-1})$ and $\mathcal{D}^*_{\scV} \mathcal{D}_{\scV} $ is an isomorphism on the $R$-component,  there is an $l_1 \in  C^{\infty}_R (X)$ such that $ \mathcal{D}^*_{\scV} \mathcal{D}_{\scV}(l_1 ) = - \psi_{R,1}$, and so  $$S(\omega_k + k^{-1} \ddb l_1 ) = c+ k^{-2} \tilde  b_1 + k^{-1} h_1   + O(k^{-2}),$$ which solves the extremal equation to order $-1$. We emphasise again that since $\tilde b_1 = k b_1 + O(1)$, $k^{-2} \tilde  b_1 $ is in fact an $O(k^{-1})$ term. Note also that we are using here that the term  $ \langle \nabla (k^{-2} \tilde b_1 + k^{-1} h_1) , \nabla k^{-1} l_1 \rangle_{\omega_k}$ is $O(k^{-2})$ and so to leading order, the potential for the holomorphic vector field appearing in the $k^{-1}$-term has not changed under the perturbation of $\omega_k$. So $$ \omega_{k,1} = \omega_k + k^{-1} \ddb l_1 $$ is the $r=1$ case of Theorem \ref{thm:extremalapprox}. 

Next we want to discuss the mapping properties of the linearisation of the scalar curvature at $\omega_{k,1}$. We note that in the proof of Proposition \ref{prop:linearisation}, we considered the Lichnerowicz operator and deduced the corresponding properties of the linearised operator because the metric was cscK to order $-1$. Now our metric is cscK to order $0$ and in general only extremal to order $-1$. Therefore the linearisation of $$\phi \mapsto  S(\omega_{k,1} + \ddc \phi )$$ will now have an extra term $$ k^{-1} \frac{1}{2} \langle \nabla \big( b_1 + h_1 \big) , \nabla \phi \rangle_{\omega_{k,1}}$$ affecting the key operators in the approximation procedure. This comes from the term $$ \frac{1}{2} \langle \nabla S(\omega_{k,1}), \nabla \phi \rangle_{\omega_{k,1}}$$ of the linearisation of the scalar curvature operator, and that $k^{-2} \tilde b_1 = k^{-1} b_1 + O(k^{-2}).$ In the notation of Proposition \ref{prop:linearisation}, this has the effect that it
\begin{enumerate}[(i)]
\item leaves $D_0$ unchanged;
\item adds an additional term $ \frac{1}{2} \langle \nabla_{\scV}  h_1  , \nabla_{\scV} \phi \rangle_{\omega_F}$ to $D_1$ on $C^{\infty}_E (X)$;
\item adds an additional term $\frac{1}{2} \langle \nabla b_1 , \nabla \phi \rangle_{\omega_B} $ to $D_2$ on $C^{\infty} (B)$.
\end{enumerate}
Note that just as in Proposition \ref{prop:linearisation}, these mapping properties remain unchanged upon any further perturbation of $\omega_{k,1}$ of the form given in Theorem \ref{thm:extremalapprox}.

We are almost read to prove Theorem \ref{thm:extremalapprox}. We first require the following.

\begin{lemma}\label{extremal-real} At an extremal symplectic connection invariant under the torus $T$ described in Remark \ref{newtorus}, the linearisation of the extremal symplectic connection operator is a real operator on $T$-invariant functions. \end{lemma}

\begin{proof} On a general K\"ahler manifold $(X,\omega)$, the Lichnerowicz operator $\scD^*\scD$ takes the form $$\scD^*\scD(\phi) =(\Delta^2 \phi + R^{\bar k j} \nabla_j \nabla_{\bar k} \phi  -  \langle i \partial \overline{\partial} \phi, \alpha \rangle) +\nabla (S(\omega))(\phi)+iJ\nabla (S(\omega))(\phi).$$ The reason this is not a real operator in general is due to the presence of the $i(\nabla S(\omega))(\phi)$ term, but if $\phi$ is taken to be invariant under the flow of $\nabla S(\omega)$, this purely imaginary term vanishes and the operator is real.

Applying this to our situation, we have a K\"ahler metric $\omega_k$ which is extremal to order $k^{-2}$. Note that the projection operator $p$ is a real operator. From what we have established above, we see that if we temporarily denote the linearisation of the purely imaginary operator for $\phi \in C^{\infty}_E(X,\R)$ by $$p(iJ\nabla_k (S(\omega_k + t\ddb \phi)) = p(iJ\nabla_k (S(\omega_k))) + p(tiJ\G(\phi)) + O(t^2),$$ then $$p(iJ\G(\phi)) = (iJ\nabla_{\scV} h_1)(\phi), $$ where $h_1$ is the global holomorphy potential constructed in Equation \eqref{h1}. Then $J\nabla_{\scV}h_1$ lies in $\Lie(T)$, and when $\phi \in C^{\infty}_E(X,\R)^{T}$ is  a $T$-invariant function, we have $(J\nabla_{\scV}h_1)(\phi) = 0$. Hence the operator is real on this function space.
\end{proof}

With this in place, we now prove Theorem \ref{thm:extremalapprox}. 
\begin{proof}[Proof of Theorem \ref{thm:extremalapprox}] From Remark \ref{newtorus}, by our assumptions $\omega_X$ and $\pi^*\omega_B$ are invariant under the compact torus $T$ we have constructed. This firstly ensures both that the linearisation of the extremal symplectic connection operator is a real operator by Lemma \ref{extremal-real}, and that the data produced will be torus invariant. Similarly $\omega_B$ is $T_q$-invariant, hence the linearised twisted extremal operator is real and the produced functions will be torus invariant by Theorem \ref{fibrations-result}. For notational convenience we drop the various superscripts which would be used to denote invariance as this plays no further role beyond from what we have just described.

We have completed the base step and now prove the inductive step. Suppose $r \geq 1$ and that we have functions $$f_1, \cdots , f_r \in C^{\infty}(B), \qquad d_1,\hdots, d_{r} \in C^{\infty}_E(X),\qquad l_1, \cdots , l_r \in C^{\infty}_R (X),$$ base holomorphy potentials  $$b_1, \cdots, b_r \in C^{\infty} (B)$$ and fibre holomorphy potentials $$h_1, \cdots, h_r \in C^{\infty}_E(X)$$ such that if we let $$\phi_{k,r} = \sum_{j=1}^{r} f_j k^{2-j}, \quad \delta_{k,r} =  \sum_{j=1}^r d_j k^{1-j}, \quad \lambda_{k,r} = \sum_{j=1}^r l_j k^{-j},$$ and $$ \eta_{k,r} = c + \sum_{j=1}^r (b_j k^{-j-1} + h_j k^{-j}),$$ then the K\"ahler form $$\omega_{k,r} = k\omega_B + \omega_X + \ddb ( \phi_{k,r} +\delta_{k,r} +  \lambda_{k,r} )$$ satisfies $$S( \omega_{k,r})  = \eta_{k,r} + \frac{1}{2} \langle \nabla \eta_{k,r}, \nabla (\phi_{k,r} +\delta_{k,r} +  \lambda_{k,r}) \rangle_{\omega_k} + O(k^{-r-1}).$$

The $O(k^{-r-1})$ error has three components
\begin{align*} S(\omega_{k,r}) =& \eta_{k,r} + \frac{1}{2} \langle \nabla \eta_{k,r}, \nabla (\phi_{k,r} +\delta_{k,r} +  \lambda_{k,r}) \rangle_{\omega_k}\\
&+ k^{-r-1}(\psi_{B,r+1} + \psi_{E,r+1} + \psi_{R,r+1}) + O(k^{-r-2})
\end{align*}
for $\psi_{B,r+1} \in C^{\infty} (B), \psi_{E,r+1} \in C^{\infty}_E (X)$ and $\psi_{R,r+1} \in C^{\infty}_R (X)$.   We begin by removing the horizontal error. Using the linearised operator on pulled back functions, we have
\begin{align*}  S(\omega_{k,r} + k^{-r+1} \ddc f_{r+1} ) =& \eta_{k,r} + \frac{1}{2} \langle \nabla \eta_{k,r}, \nabla (\phi_{k,r} +\delta_{k,r} +  \lambda_{k,r}) \rangle_{\omega_k}\\
&+ k^{-r-1}\big( \psi_{B,r+1} - \mathcal{L}_{\alpha} (f_{r+1}) +   \frac{1}{2} \langle \nabla b_1, \nabla f_{r+1} \rangle \big)\\
&+ k^{-r-1}( \psi'_{E,r+1} + \psi'_{R,r+1}) + O(k^{-r-2})
\end{align*} for some $\psi'_{E,r+1} \in C^{\infty}_E (X)$ and $\psi'_{R,r+1} \in C^{\infty}_R (X)$ depending on $f_{r+1}$. Since $\mathcal{L}_{\alpha}$ is invertible modulo holomorphy potentials on $B$, there is a base holomorphy potential $b_{r+1}$ and choice of $f_{r+1}$ such that $\psi_{B,r+1} - \mathcal{L}_{\alpha} (f_{r+1})  = b_{r+1}$. Thus with this choice of $f_{r+1}$ we have
\begin{align*}  S(\omega_{k,r} + k^{-r+1} \ddc f_{r+1} ) =& \eta_{k,r} + \frac{1}{2} \langle \nabla \eta_{k,r}, \nabla (\phi_{k,r} +\delta_{k,r} +  \lambda_{k,r}) \rangle_{\omega_k}\\
&+ k^{-r-1}\big( b_{r+1} +   \frac{1}{2} \langle \nabla b_1, \nabla f_{r+1} \rangle_{\omega_B} \big)\\
&+ k^{-r-1}( \psi'_{E,r+1} + \psi'_{R,r+1}) + O(k^{-r-2})
\end{align*} 

We now note that since $b_{r+1}$ is pulled back from $B$, we have $$\langle \nabla k^{-r-1} b_{r+1}  , \nabla  (\phi_{k,r} +\delta_{k,r} +  \lambda_{k,r})  \rangle_{\omega_k} = O(k^{-r-2}),$$ and so we obtain $$\langle \nabla k^{-r-2} \tilde b_{r+1}  , \nabla  (\phi_{k,r} +\delta_{k,r} +  \lambda_{k,r})  \rangle_{\omega_k} = O(k^{-r-2}),$$ from the properties of the lift. The same holds when $\phi_{k,r} $ is replaced by $\phi_{k,r+1} = \phi_{k,r} + k^{-r+1} f_{r+1}.$ Moreover, we also have that 
\begin{align*} k^{-r-1} \langle \nabla b_1, \nabla f_{r+1} \rangle_{\omega_B} &= \langle \nabla k^{-1} b_1, \nabla k^{-r+1} f_{r+1}  \rangle_{\omega_k} + O(k^{-r-2}) \\
&= \langle \nabla k^{-2} \tilde b_1, \nabla k^{-r+1} f_{r+1}  \rangle_{\omega_k} + O(k^{-r-2}).
\end{align*}The upshot is that $$\frac{1}{2} \langle \nabla \eta_{k,r}, \nabla (\phi_{k,r} +\delta_{k,r} +  \lambda_{k,r}) \rangle_{\omega_k} + k^{-r-1}\big( b_{r+1} +   \frac{1}{2} \langle \nabla b_1, \nabla f_{r+1} \rangle \big)$$ equals $$\frac{1}{2} \langle \nabla \eta'_{k,r+1}, \nabla (\phi_{k,r+1} +\delta_{k,r} +  \lambda_{k,r}) \rangle_{\omega_k}$$ up to order $k^{-r-2},$ where $\eta'_{k,r+1} = \eta_{k,r} + k^{-r-2} \tilde b_{r+1}.$

Letting $\omega'_{k,r+1} = \omega_{k,r} + k^{-r+1} \ddc f_{r+1}$, we so far have that 
\begin{align*}  S( \omega_{k,r+1}'  ) =& \eta'_{k,r+1} + \frac{1}{2} \langle \nabla \eta'_{k,r+1}, \nabla (\phi_{k,r+1} +\delta_{k,r} +  \lambda_{k,r}) \rangle_{\omega_k}\\
&+ k^{-r-1}( \psi'_{E,r+1} + \psi'_{R,r+1}) + O(k^{-r-2}).
\end{align*}
We now remove the error in the $E$-component. From the linearisation, we obtain
\begin{align*}S\big( \omega'_{k,r+1} + k^{-r}  \ddc (\delta_{r+1} )\big)  =& \eta'_{k,r+1}+ \frac{1}{2} \langle \nabla \eta'_{k,r+1}, \nabla (\phi_{k,r+1} +\delta_{k,r} +  \lambda_{k,r}) \rangle_{\omega_k} \\
&+ k^{-r-1}\big( \psi'_{E,r+1} - p \circ \L_1 (d_{r+1} ) + \frac{1}{2} \langle \nabla_{\scV}  h_1  , \nabla_{\scV} d_{r+1} \rangle_{\omega_F} )  \big) \\
& +  k^{-r-1}\big( \psi''_{R,r+1} \big) + O(k^{-r-2}),\end{align*} for some $\psi''_{R,r+1} \in C^{\infty}_R (X)$ depending on $d_{r+1}$ and $f_{r+1}$.   

From Theorem \ref{thm:mappingprops}, the operator $p \circ \L_1$ is invertible on $E$ modulo global fibrewise holomorphy potentials, and so we can pick $d_{r+1}$ and a global holomorphy potential $h_{r+1}$ such that $\psi'_{E,r+1} = p \circ \L_1 (d_{r+1}) + h_{r+1}$. With this choice of $d_{r+1}$, we have 
\begin{align*}S\big(\omega'_{k,r+1} + k^{-r}  \ddc (d_{r+1} )\big)  =& \eta'_{k,r+1}+ \frac{1}{2} \langle \nabla \eta'_{k,r+1}, \nabla (\phi_{k,r+1} +d_{k,r} +  \lambda_{k,r}) \rangle_{\omega_k} \\
&+ k^{-r-1}\big( h_{r+1} + \frac{1}{2} \langle \nabla_{\scV}  h_1  , \nabla_{\scV} d_{r+1} \rangle_{\omega_F} )  \big) \\
& +  k^{-r-1}\big( \psi''_{R,r+1} \big) + O(k^{-r-2}),\end{align*}

We now proceed similarly to the case of the base component to show that this is actually of the required form, up to the $\psi''_{R,r+1}$ error. We have that $\delta_{k,r}+  \lambda_{k,r}$ is $O(k^{-1})$, and so $\langle \nabla k^{-r-1} h_{r+1} , \nabla (\delta_{k,r}+  \lambda_{k,r}) \rangle_{\omega_k}$ is $O(k^{-r-2}).$  While $\phi_{k,r+1}$ is only $O(1)$, it is pulled back from the base, so when taking the inner product we have that $\langle \nabla k^{-r-1} h_{r+1} , \nabla\phi_{k,r+1} \rangle_{\omega_k}$ is $O(k^{-r-2})$,  too. Similarily when adding $k^{-r} d_{r+1}$ to $\delta_{k,r}$ to form $\delta_{k, r+1}.$  Moreover, we also have that 
\begin{align*}k^{-r-1}  \langle \nabla_{\scV}  h_1  , \nabla_{\scV} d_{r+1} \rangle_{\omega_F} &= \langle \nabla k^{-1} h_1, \nabla k^{-r} d_{r+1}  \rangle_{\omega_k} + O(k^{-r-2}) ,
\end{align*}
and so we obtain that $$\frac{1}{2} \langle \nabla \eta_{k,r}', \nabla (\phi_{k,r} +\delta_{k,r} +  \lambda_{k,r}) \rangle_{\omega_k} + k^{-r-1}\big( h_{r+1} +   \frac{1}{2} \langle \nabla_{\scV}  h_1  , \nabla_{\scV} d_{r+1} \rangle_{\omega_F}  \big)$$ equals $$\frac{1}{2} \langle \nabla \eta_{k,r+1}, \nabla (\phi_{k,r+1} +\delta_{k,r+1} +  \lambda_{k,r}) \rangle_{\omega_k}$$ up to order $k^{-r-2},$ where $\eta_{k,r+1} = \eta_{k,r+1}' + k^{-r-1} h_{r+1}$.

All that is left is the error coming from $\psi''_{R,r+1}$, since we have shown that 
\begin{align*}  S\big( \omega''_{k,r}  \big) =& \eta_{k,r+1} + \frac{1}{2} \langle \nabla \eta_{k,r+1}, \nabla (\phi_{k,r+1} +\delta_{k,r+1} +  \lambda_{k,r}) \rangle_{\omega_k} \\
&+ k^{-r-1}\psi''_{R,r+1} + O(k^{-r-2}),
\end{align*}
where $ \omega''_{k,r} =\omega_{k,r} + \ddc ( k^{-r+1} f_{r+1} + k^{-r} d_{r+1} ).$ 
Since the vertical Lichnerowicz operator is invertible on $R$, there is an $l_{r+1}$ such that $ - \D^*_{\scV}\D_{\scV} (l_{r+1}) = \psi''_{R,r+1}$, which removes $\psi''_{R,r+1}$. The additional term coming from the inner product does not matter here as $\langle \nabla \eta_{k, r+1} , \nabla k^{-r-1} l_{r+1} \rangle_{\omega_k}$ is $O(k^{-r-2}).$
\end{proof}

\begin{remark} There is an error in the analogous approximation scheme used by Br\"onnle \cite[Section 3]{bronnle}. He corrects the $C^{\infty}(B)$ term last, but this may introduce a new $C^{\infty}_0(X)$ error at the same order of $k$, which would be problematic. The order we use seems to be the unique order which makes the argument work. \end{remark}

\section{Solving the non-linear equation}\label{sec:implicit}

We now perturb our approximately extremal metric constructed in Section \ref{sec:approx} to a genuine extremal metric. The techniques we employ are very similar to the previous work \cite{fine1,bronnle,fibrations}; we briefly give the details as there are some minor differences related to the geometry of our situation.

\begin{theorem}\label{endmainthm}Suppose $q: (B,L)\to \M$ admits a twisted extremal metric $\omega_B$ and $\pi: (X,H) \to (B,L)$ admits an extremal symplectic connection $\omega_X$. Suppose in addition that all automorphisms of the moduli map $q: (B,L) \to \M$ lift to $(X,H)$. Then there exists an extremal metric in the class $kc_1(L)+c_1(H)$ for all $k \gg 0$. \end{theorem}

The proof will rely on a quantitative version of the implicit function theorem, of the following form:

\begin{theorem}\label{implicitfnthm}\cite[Theorem 25]{bronnle} Consider a differentiable map of Banach spaces $\F: \scB_1 \to \scB_2$  whose derivative at $0$ is surjective with right-inverse $\scP$. Denote by 
\begin{enumerate}[(i)]
\item $\delta'$ the radius of the closed ball in $\scB_1$ centred at $0$ on which $\F-D\F$ is Lipschitz of constant $(2\|\scP\|)^{-1}$,
\item $\delta = \delta' (2\|\scP\|)^{-1}$.
\end{enumerate} 
For all $y \in B_2$ such that $\|y-\F(0)\| < \delta$, there exists $x\in \scB_1$ satisfying $\F(x) = y$. 
\end{theorem}

We will apply this with $\F$ the extremal operator. To produce an extremal metric, it follows that that we need to bound both the nonlinear terms of the extremal operator and the right inverse $P$ of its linearisation. 

We fix a large integer $r \gg 0$, and start with a bound on the linearisation of the extremal operator. Denote by $L^2_p(\omega_{k,r})$ the Sobolev space of functions on $X$ measured with respect to $\omega_{k,r}$ which have integral zero with respect to $\omega_{k,r}$. We shall assume $p \gg 0$, and we emphasise that this $p$ is unrelated to the projection onto $C_E^{\infty}(X)$ of Section \ref{sec:optimal}. 

We begin with the following estimate on the Lichnerowicz operator due to Fine.

\begin{lemma} \cite[Lemma 6.8]{fine1} For each $p$ and $r$, there are constants $C, k_0>0$ such that for all $k \geq k_0$ and $\phi \in \|\phi\| \in L^2_{p + 4}(\omega_{k,r})$ we have $$\|\phi\|_{L^2_{p + 4}(\omega_{k,r})} \leq C ( \|\phi\|_{L^2(\omega_{k,r})}  + \|\scD_{k,r}^*\scD_{k,r}\phi\|_{L^2_{p}(\omega_{k,r})}),$$ where $\scD_{k,r}^*\scD_{k,r}$ denotes the Lichnerowicz operator with respect to $\omega_{k,r}$. \end{lemma}

Fine proves this for general fibrations endowed with metrics of the form we have considered, and does not use any special geometric properties that do not apply to our situation. In our situation, however, the Lichnerowicz operator will have a non-trivial kernel and cokernel when $(X,H)$ admits automorphisms. To describe the kernel and cokernel more explicitly, we begin by describing the change in holomorphy potentials as $k$ and $r$ vary.

We may assume $\omega_X$ is K\"ahler, rather than just relatively K\"ahler, as this does not change any of the hypotheses. Indeed, if one modifies $\omega_X$ by pulling back, say, the twisted extremal metric from $B$, this does not change that the resulting form is fibrewise cscK or an optimal (or extremal) symplectic connection.

Following Remark \ref{newtorus}, we fix a  torus containing the flow of the real holomorphic vector field extremal symplectic connection, and a lift of a torus containing the flow of the real holomorphic vector field associated to the twisted extremal metric on the base. As in Remark \ref{newtorus}, we fix a torus $T$ containing both of these tori, and emphasise again that if $\omega_X$ is an optimal symplectic connection and $\omega_B$ is a twisted cscK metric, $T$ can be taken to be trivial. 

We also fix a maximal compact torus $T_q\subset \Aut(q)$ containing the flow of the twisted extremal vector field, and a maximal compact torus of $\Aut(X,H)$ containing the vector field associated to the optimal symplectic connection. Choosing a lift of $T_q$, we denote by $T_X \subset \Aut(X,H)$ a maximal compact torus containing both $T_q$ and $T_E$. By Proposition \ref{automsofX}, the lift of $T_q$ determines a splitting of of $\Lie(T_X)$ as a direct sum of $\Lie(T_q)$ and $\Lie(T_E)$, and similarly for the natural complexifications. 

Let $\xi \in T_X^{\C}\subset \Aut(X,H)$ be a vector field on $(X,H)$. Let $\xi_q + \xi_E$ be the decomposition of $\xi$ with $\xi_q$ the lift of a vector field on $q: (B,L) \to \M$ and $\xi_E$ induced by an element of $\Lie(X/B,H)$ using the decomposition just described. Let $h_X$ be the holomorphy potential of $\xi$ with respect to $\omega_X$ and $h_B$ the holomorphy potential of $\xi_B$ with respect to $\omega_B$. Thus $k\pi^*h_B + h_X$ is the holomorphy potential for $\xi$ with respect to $\omega_k = k\omega_B +\omega_X$, where we have used that the flow of $\xi_E$ is induced by an element of $\Lie(X/B,H)$. 

We denote by $\gamma_{k,r}$ the K\"ahler potential for $\omega_{k,r}$ with respect to $\omega_k$ constructed in Section \ref{sec:approx}, by which we mean $$\omega_{k,r} = \omega_k + \ddb \gamma_{k,r}.$$ Then the holomorphy potential for $\xi$ with respect to $\omega_{k,r}$ is $$k\pi^*h_B + h_X + \frac{1}{2}\langle \nabla \phi_{k,r}, \nabla(k\pi^*h_B + h_X)\rangle,$$ where the gradient is computed with respect to $\omega_k$.

\begin{definition} For each $k,r$ we define a map $$\tau_{r,k}: \Lie(\Aut(X,H)) \to C^{\infty}(X)$$ by $$\tau_{r,k}(\xi) = k\pi^*h_B + h_X + \frac{1}{2}\langle \nabla \phi_{k,r}, \nabla(k\pi^*h_B + h_X)\rangle.$$ \end{definition}

The map $\tau_{k,r}$ induces a map from $\mft = \Lie(T_X)$ to $C^{\infty}(X,\R)$.

\begin{lemma}\label{inversebounds} There exists a constant $C$ independent of $k$ and $\phi$ such that the operators $\L_{k,r}: (L^2_{p+4}(\omega_{k+4,r}))^T\times \bar\mft \to (L^2_p(\omega_{k,r}))^T$ defined by $$\L_{k,r}(\phi, \xi) = -\scD_{k,r}^*\scD_{k,r}(\phi) + \tau_{k,r}(\xi)$$ have right inverses $\scQ_{k,r}$ which satisfy $$\|\scQ_{k,r}(\phi)\|_{L^2_p(\omega_{k,r})}\leq C k^3\|\phi\|_{L^2_p(\omega_{k,r})}.$$  
\end{lemma}

\begin{proof} See \cite[Lemma 6.4]{fibrations}. The only new part comes from the global fibrewise holomorphy potentials, but restricted to these functions the map $\mathcal{L}_{k,r}$  is $O(1)$, hence so is the inverse, giving a better bound than than the $O(k^3)$ bound required.
\end{proof}

The operator $\D^*_{k,r}\D_{k,r}$ is in general a complex operator. We now, however, consider the linearisation $\scP_{k,r}$ of the actual extremal operator, which is necessarily real, being the linearised operator of a real-valued operator. Moreover, to leading order the linearised operator  $\scP_{k,r}$  of the genuine extremal operator is the Lichnerowicz type operator, and so we can use bounds on the operator norm of the former operator to obtain the desired bounds for the latter.

\begin{proposition}\label{prop:bounds1}\cite[Proposition 6.5]{fibrations} Fix a positive integer $r \gg 0$. Denote \begin{align*} & \scG_{k,r} : (L_{p+4}^2)^T \times \bar\mft  \to (L_p^2)^T, \\ & \scG_{k,r}(\phi,h) = -\mathcal{D}^*_{k,r} \mathcal{D}_{k,r} (\phi) + \frac{1}{2} \langle \nabla \left( S(\omega) - \tau_{k,r}(h) \right), \nabla \phi \rangle -\tau_{k,r}(h).\end{align*} Then there exists a $C$ independent of $k$ such that $\scG_{r,p}$ has a right inverse $\scP_{k,r}$ with $\|\scP_{r,p}\|_{op,k,r} \leq C k^3$, where $\|\cdot \|_{op,k,r} $ denotes the operator norm with respect to $\omega_{k,r}$. 
\end{proposition}

The next required bound concerns the non-linear part of the extremal operator. Precisely, we will consider the operator $\F_{k,r}:   (L_{p+4}^2)^T \times \bar\mft  \to (L_p^2)^T$ defined by  $$\F_{k,r}(\psi, h)=S(\omega_{k,r} + \ddb \psi) -  \frac{1}{2}\langle \nabla \eta_{k,r},\nabla (\phi_{k,r})\rangle - \eta_{k,r}  -\frac{1}{2} \langle \nabla(\tau_{k,r} (h)) , \nabla \psi \rangle -  \tau_{k,r} (h),$$ where the gradients are taken with respect to $\omega_k$ and $\eta_{k,r}$ is the holomorphy potential constructed in Section \ref{sec:approx} which makes $\omega_{k,r}$ approximately extremal. Note that the linearisation of $\F_{k,r}$ is the operator $G_{k,r}$ considered above, and a zero of $\F_{k,r}$ is a K\"ahler potential for an extremal metric. Moreover, $\frac{1}{2}\langle \nabla \eta_{k,r},\nabla (\phi_{k,r})\rangle - \eta_{k,r} $ are constants independent of the input $(\psi,h)$, and are subtracted simply so that $\mathcal{F}_{k,r} (0,0)$ is close to $0$. 

Below we shall take $p \gg 0$ so that $L^2_p$ embeds in $C^{4,\alpha}$ to apply work of Fine, and eventually to conclude that a $L^2_p$-extremal metric is actually smooth, as follows for example from \cite[Lemma 55]{bronnle}.

\begin{lemma}\label{lemma:bound-nonlinear}

Denote by $\scN_{k,r} = \F_{k,r} - \scG_{k,r} $ the nonlinear part of the extremal operator. There exist constants $c,C >0$ such that for all $k \gg 0$, if $\phi,\psi \in (L^2_{p+4}(\omega_{k,r}))^T$ satisfy $\| \phi \|_{L^2_{p+4}(\omega_{k,r})} , \| \psi \|_{L^2_{p+4}} \leq c$ then $$ \| \scN_{k,r} (\phi) - \scN_{k,r} (\psi) \|_{L^2_p } \leq C \big( \| \phi \|_{L^2_{p+4}(\omega_{k,r})} + \| \psi \|_{L^2_{p+4}(\omega_{k,r})} \big)  \| \phi - \psi \|_{L^2_{p+4}(\omega_{k,r})}.$$

\end{lemma}

\begin{proof} The proof is an application of the Mean Value Theorem; we refer to \cite[Lemma 2.7]{fine1} and \cite[Lemma 6.6]{fibrations} for further details.
\end{proof}

We now proceed to the proof of Theorem \ref{endmainthm}. 

\begin{proof}[Proof of Theorem \ref{endmainthm}] We apply the qualitative inverse function theorem to the operators $\F_{k,r}$.

The first part of the statement of the quantitative implicit function theorem requires a bound on the non-linear operators $\scN_{k,r} = \F_{k,r} - D\F_{k,r}$. Lemma \ref{lemma:bound-nonlinear} provides a $C>0$ such that for all balls of radius $0<\lambda \ll 1$, the operator $\scN_{k,r}$ is Lipschitz on the ball of radius $\lambda$ with Lipschitz constant $\lambda C$. It follows from this and Lemma \ref{inversebounds} that for $k \gg 0$, the radius $\delta_k'$ for which the operator $\scN_{k,r}$ is Lipschitz with constant $(2\|\scP_{k,r}\|)^{-1}$ is bounded below by $C'k^{-3}$ for some constant $C'$. Thus $\delta_k = \delta_k'(2\|\scP_{k,r}\|)^{-1}$ is bounded below by $C'' k^{-6}$ for some constant $C''$.

Next we require a bound on $\F_{k,r}(0,0)$. This is provided by Theorem \ref{thm:extremalapprox}, which provides that $\F_{k,r} = O(k^{-r-1})$. As mentioned in Remark \ref{norminuse}, the proof given in Section \ref{sec:approx} provides this statement \emph{pointwise}, whereas for our application we require estimates in some $C^p$ norm. However, work of Fine \cite[Lemma 5.6 and Lemma 5.7]{fine1} directly and without change allows us to pass from pointwise norms to the analogous bound in the $C^p$-norm for any $p>0$. 

Thus from Fine's work we obtain that with respect to the $C^p$-norm, we have a bound $\F_{k,r}(0,0) = O(k^{-r-1})$. Note that the $\eta_{k,r}$ appearing in the statement of Theorem \ref{thm:extremalapprox} have terms $\tilde b_j$ that depend on $k$, but that in fact we could have rewritten the expansion in terms of fixed functions, using the expression for the lifts of the holomorphy potentials $b_j$ with respect to $\omega_k$. That these functions are independent of $k$ is important in order to go from the local to global estimates. In terms of the $L^2_p(\omega_{k,r})$-norm, a result of Fine gives that $\|\F_{k,r}(0,0)\|_{L^2_p(\omega_{k,r})} = O(k^{-r-\frac{1}{2}})$ \cite[Lemma 5.7]{fine1}. Thus provided $r\geq 6$, we obtain an $L^2_p(\omega_{k,r})$-bound on $\F_{k,r}(0,0)$ of order $k^{-6-\frac{1}{2}}$.

Thus the hypotheses of the qualitative implicit function are satisfied, and hence $(X,kL+H)$ admits an extremal metric for all $k \gg 0$.
\end{proof}

\vspace{4mm} \noindent {\bf Acknowledgements:} A large part of the present work was done while the authors participated in the ICMS Research in Groups program. We are very grateful to the ICMS for the opportunity. We are also very grateful to the LMS for funding a visit of the second author to the first at the University of Cambridge through their Reseach in Pairs Scheme, where another large part of the present work was completed. In addition, the second author received funding from CIRGET. Finally we would like to thank the anonymous referee for helpful comments.

 \bibliography{morefibrations}
\bibliographystyle{amsplain}

\end{document}